\documentclass[12pt]{article}
\usepackage{amsmath}
\usepackage{latexsym}
\usepackage{amsfonts}
\usepackage{euler}
\usepackage{graphicx}

\def\Bbb{\mathbb}

\def\eea{\end{eqnarray*}}

\newtheorem{main}{Theorem}

\newtheorem{defn}{Definition}
\newtheorem{thm}{Theorem}
\newtheorem{prop}[thm]{Proposition}
\newtheorem{cor}[thm]{Corollary}
\newtheorem{lem}[thm]{Lemma}
\newtheorem{conj}[thm]{Conjecture}

\newtheorem{pro}[thm]{Problem}

\newenvironment{proof}{\medskip \noindent
{\bf Proof.}}{\hfill \rule{.5em}{1em}
\\}

\newenvironment{rmk}{\mbox{ }\\{\bf  Remark}\mbox{ }}{
\hfill $\Box$\mbox{}\bigskip}

\def\CP{{\mathbb C \mathbb P}}
\def \x {\times}
\def \eu{{\text{e}}}
\def \sign{{\text{sign}}}

\newcommand{\CPb}{\overline{\mathbb{CP}}{}^{2}}

\newcommand{\R}{\mathbb{R}}
\newcommand{\Z}{\mathbb{Z}}
\newcommand{\Q}{\mathbb{Q}}

\begin{document}

\title{Families of $4$-Manifolds with Nontrivial Stable Cohomotopy Seiberg-Witten Invariants, and Normalized Ricci Flow}

\author{R. \.Inan\c{c} Baykur and Masashi Ishida} 

\date{}

\maketitle

\begin{abstract}
In this article, we produce infinite families of $4$-manifolds with positive first betti numbers and meeting certain conditions on their homotopy and smooth types so as to conclude the non-vanishing of the stable cohomotopy Seiberg-Witten invariants of their connected sums. Elementary building blocks used in \cite{ishi-1} are shown to be included in our general construction scheme as well. We then use these families to construct the first examples of families of closed smooth $4$-manifolds for which Gromov's simplicial volume is nontrivial, Perelman's $\bar{\lambda}$ invariant is negative, and the relevant Gromov-Hitchin-Thorpe type inequality is satisfied, yet no non-singular solution to the normalized Ricci flow for any initial metric can be obtained. In \cite{fz-1}, Fang, Zhang and Zhang conjectured that the existence of any non-singular solution to the normalized Ricci flow on smooth $4$-manifolds with non-trivial Gromov's simplicial volume and negative Perelman's $\bar{\lambda}$ invariant implies the Gromov-Hitchin-Thorpe type inequality. Our results in particular imply that the converse of this fails to be true for vast families of $4$-manifolds.
\end{abstract}

%%%%%%%%%%%%%%%%%%%%%%%%%%%%%%%%%%%%%%%%%%%%%%%%%%%%%%%%%%%%%%%%%%%%%%%%%%
\section{Introduction }\label{sec-0}
%%%%%%%%%%%%%%%%%%%%%%%%%%%%%%%%%%%%%%%%%%%%%%%%%%%%%%%%%%%%%%%%%%%%%%%%%%

Let $X$ be a closed smooth Riemannian $4$-manifold $X$ with $b^{+}(X)>1$, where ${b}^+(X)$ denotes the dimension of the maximal positive definite linear subspace in the second cohomology of $X$. In what follows, $\eu(X)$ and $\sign(X)$ denote respectively the Euler characteristic and signature of $X$. Recall that a spin${}^{c}$-structure $\Gamma_{X}$ on $X$ induces a pair of spinor bundles ${S}^{\pm}_{\Gamma_{X}}$ which are Hermitian vector bundles of rank 2 over $X$. A Riemannian metric on $X$ and a unitary connection $A$ on the determinant line bundle ${\cal L}_{\Gamma_{X}} := det({S}^{+}_{\Gamma_{X}})$ induce the twisted Dirac operator ${\cal D}_{{A}} : \Gamma({S}^{+}_{\Gamma_{X}}) \longrightarrow \Gamma({S}^{-}_{\Gamma_{X}})$. The Seiberg-Witten monopole equations \cite{w} over $X$ are the following non-linear partial differential equations for a unitary connection $A$ of the complex line bundle ${\cal L}_{\Gamma_{X}}$ and a spinor $\phi \in \Gamma({S}^{+}_{\Gamma_{X}})$:
\begin{eqnarray*}
{\cal D}_{{A}}{\phi} = 0, \ {F}^{+}_{{A}} = iq(\phi),  
\end{eqnarray*}
here ${F}^{+}_{{A}}$ is the self-dual part of the curvature of $A$ and $q : {S}^{+}_{\Gamma_{X}} \rightarrow {\wedge}^{+}$ is a certain natural real-quadratic map, where ${\wedge}^{+}$ is the bundle of self-dual 2-forms. The quotient space of the set of solutions to the Seiberg-Witten monopole equations by gauge group is called the Seiberg-Witten moduli space. In his celebrated article \cite{w}, Witten introduced an invariant of smooth $4$-manifolds by using the fundamental homology class of the Seiberg-Witten moduli space, which is now called the \textit{Seiberg-Witten invariant}, and is well-defined for any closed $4$-manifold $X$ with $b^{+}(X)>1$. The Seiberg-Witten invariant defines an integer valued function $SW_{X}$ over the set of all isomorphism classes of spin${}^c$ structures of $X$ with $b^{+}(X)>1$. \par 

More recently, Bauer and Furuta \cite{b-f, b-1} adopted a remarkable approach to introduce a refinement of $SW_{X}$ without using the Seiberg-Witten moduli space. They introduced a new invariant, which takes values in a certain stable cohomotopy group ${\pi}^{{b}^+}_{S^1, \mathcal{B}}({\rm Pic}^0(X), {\rm ind} D)$, where ${b}^+:={b}^+(X)$  and ${\rm ind} D$ is the virtual index bundle for the Dirac operators parametrized by the $b_{1}(X)$-dimensional Picard torus ${\rm Pic}^0(X)$. This invariant is called the \textit{stable cohomotopy Seiberg--Wittten invariant}, and herein will be denoted as: 
\begin{eqnarray*}\label{b-f-inv}
BF_{X}(\mathfrak{s}) \in {\pi}^{{b}^+}_{S^1, \mathcal{B}}({\rm Pic}^0(X), {\rm ind} D).
\end{eqnarray*}
Moreover, in \cite{b-1} Bauer proved a non-vanishing theorem of $BF_{*}$ for a connected sum of $4$-manifolds with $b^{+} > 1$ and $b_{1} =0$ \cite{b-1} subject to a couple of conditions (See the paragraph following Theorem \ref{new-BF-non-vanishing} below for the precise conditions), and used this theorem to show that there are  $4$-manifolds that appear as such connected sums, for which $SW_{*}$ is trivial but $BF_{*}$ is not. In particular, $BF_{*}$ is a strictly stronger invariant than $SW_{*}$. \par

In \cite{ishi-1}, H. Sasahira and the second author of the current article generalized Bauer's non-vanishing theorem by removing the condition $b_{1}=0$ for the summands. For this new non-vanishing theorem, which is now formulated for $4$-manifolds with arbitrary $b_1$, the following definition was needed to constrain the cohomology group of the $4$-manifold:

\begin{defn}[\cite{ishi-1}]\label{BF-third-condition}
Let $X$ be any closed oriented smooth $4$-manifold with ${b}^{+}(X)>1$. Let $\Gamma_{X}$ be a spin${}^{c}$ structure on $X$. Let ${c}_{1}({\cal L}_{{\Gamma}_{X}})$ be the first Chern class of the complex line bundle ${\cal L}_{{\Gamma}_{X}}$ associated with $\Gamma_{X}$. Finally, let ${\frak e}_{1}, {\frak e}_{2}, \cdots, {\frak e}_{s}$ be a set of generators of ${H}^{1}(X, {\mathbb Z})$, where $s={b}_{1}(X)$. Then, define 
\begin{eqnarray*}\label{spin-01}
\frak{S}^{ij}(\Gamma_X):=\frac{1}{2}< c_1(L_{\Gamma_X}) \cup \frak{e}_i \cup \frak{e}_j, [X] >, 
\end{eqnarray*}
where $[X]$ is the fundamental class of $X_i$ and $<\cdot, \cdot >$ is the pairing between cohomology and homology. 
\end{defn}

\noindent We can now introduce the notion of \textit{BF-admissibility} for a $4$-manifold, as discussed in \cite{ishi-1}:
\begin{defn}\label{BF-adm-def}
A closed oriented smooth $4$-manifold $X$ with $b^{+}(X) > 1$ is called BF-admissible if the following three conditions are satisfied.
\begin{enumerate}
\item There exists a spin${}^{c}$-structure $\Gamma_{X}$ with $SW_{X}(\Gamma_{X}) \equiv 1 \ (\ \bmod \ 2)$ and \linebreak $c^{2}_{1}({\cal L}_{\Gamma_{X}}) = 2 \eu(X) + 3 \sign(X)$, where $c_{1}({\cal L}_{\Gamma_{X}})$ is the first Chern class of ${\cal L}_{\Gamma_{X}}$. 
\item ${b}^{+}(X)-{b}_{1}(X) \equiv 3 \ (\bmod \ 4)$. 
\item $\frak{S}^{ij}(\Gamma_{X}) \equiv 0 \ (\bmod \ 2) \ \text{for all i, j}$.
\end{enumerate}  
\end{defn}
Notice that, under the first condition in Definition \ref{BF-adm-def}, any 4-manifold $X$ possesses an almost complex structure $J$ with $c_{1}(X, J) = c_{1}({\cal L}_{\Gamma_{X}})$ by a result of Wu \cite{wu}. Hence, any BF-admissible 4-manifold must be almost complex. \par
The new non-vanishing theorem for the stable cohomotopy Seiberg-Witten invariant can be stated as follows: 
\begin{thm}[\cite{ishi-1}]\label{new-BF-non-vanishing}
For $i=1,2,3$, let $X_i$ be a BF-admissible, closed oriented smooth $4$-manifold. Then the connected sum $\#^{j}_{i=1}{X}_{i}$ has a non-trivial stable cohomotopy Seiberg-Witten invariant, where $j=2,3$. 
\end{thm} 

Observe that, when ${b}_{1}(X_{i}) = 0$, the second condition for BF-admissibility just reads as ${b}^{+}(X) \equiv 3 \ (\bmod \ 4)$ and the third one holds trivially. That is, Theorem \ref{new-BF-non-vanishing} when ${b}_{1}(X_{i}) = 0$ for all summands is nothing but Bauer's non-vanishing theorem from \cite{b-1}, and therefore can be regarded as a natural generalization of the latter. \par

In order to apply this new non-vanishing theorem of stable cohomotopy Seiberg-Witten invariant to geometry and topology of smooth $4$-manifolds, it is essential to find BF-admissible $4$-manifolds. Of particular interest was to find BF-admissible $4$-manifolds with $b_1 \neq 0$, so as to get new applications that does not follow from Bauer's original non-vanishing theorem stated for $b_1=0$. In \cite{ishi-1}, two types of $4$-manifolds were seen to be BF-admissible: Products $\Sigma_{g} \times \Sigma_{h}$ of two Riemann surfaces of odd genera, and primary Kodaira surfaces. Failing to get other examples of $4$-manifolds with $b_{1} > 0$ satisfying the BF-axioms, the authors raised the following problem in the same work \cite{ishi-1}:

\begin{pro}[Problem 75 in \cite{ishi-1}]\label{pro-75}
Find BF-admissible, closed oriented $4$-manifolds with $b_{1}> 0$, which are not primary Kodaira surfaces or products $\Sigma_{g} \times \Sigma_{h}$ of Riemann surfaces with odd genera.   
\end{pro}

In the first part of our article, we will answer this problem by showing the existence of vast families of BF-admissible $4$-manifolds with $b_{1} > 0$. Moreover, we will see that these families naturally include products $\Sigma_{g} \times \Sigma_{h}$ and primary Kodaira surfaces. The main surgical operation involved in these constructions is the \textit{Luttinger surgery} along Lagrangian tori \cite{Lut}, defined and discussed in detail in Subsection \ref{luttinger} below. \par

In Subsection \ref{sec-2}, we will introduce the notion of \textit{surgered product manifolds} which are obtained from products $\Sigma_{g} \times \Sigma_{h}$ via Luttinger surgeries along certain homologically essential Lagrangian tori. Note that $\Sigma_{g} \times \Sigma_{h}$ are the trivial examples of surgered product manifolds. We will prove that: 

\begin{main}\label{main-A}
Let $\Sigma_{g} \times \Sigma_{h}$ be the product of two Riemann surfaces of odd genera $g, h$, equipped with the product symplectic form. Then any surgered product manifold obtained from $\Sigma_{g} \times \Sigma_{h}$ with $b_{1} > 0$ is BF-admissible. Moreover, and primary Kodaira surface is a surgered product manifold obtained from $T^2 \x T^2$, and is BF-admissible.
\end{main}

In \cite{ABBKP}, Akhmedov, Baldridge, Kirk, D. Park, and the first author of the current article, showed that a very large portion of the symplectic geography plane could be populated with minimal symplectic $4$-manifolds. In Subsection \ref{sec-3}, we will make use of these examples, while paying attention to preserving BF-admissibility during the employed surgical operations, to prove the following:

\begin{main}\label{main-B}
Let $a$ and $b$ are integers satisfying $2 a + 3 b \geq 0$, $a + b \equiv 0 \ (\bmod \ 8)$, and $b < -1$ is satisfied. Set as $\alpha = {(a + b)}/{2}$ and $\beta = {(a - b)}/{2}$. Then, there exists a BF-admissible, irreducible symplectic $4$-manifold with fundamental group $\mathbb Z$ which is homeomorphic to 
\begin{eqnarray}\label{irr-home-1}
\alpha {\mathbb C}{P}^{2} \# \beta \overline{{\mathbb C}{P}^{2}} \# ({S}^{1} \times {S}^{3})
\end{eqnarray}
and a BF-admissible, irreducible symplectic $4$-manifold with fundamental group ${\mathbb Z}_{p}$, $p$ odd, which is homeomorphic to 
\begin{eqnarray}\label{irr-home-2}
(\alpha -1) {\mathbb C}{P}^{2} \# (\beta -1) \overline{{\mathbb C}{P}^{2}} \# Y_{p}, 
\end{eqnarray}
where $Y_{p}$ is the $4$-manifold with fundamental group ${\mathbb Z}_{p}$, obtained from the product $L(p, 1) \times S^{1}$ of Lens space $L(p, 1)$ and $S^{1}$ after a $0$ surgery along $\{ pt \} \times S^{1}$. 
\end{main}

Note that these symplectic $4$-manifolds are not brand new; they are produced using the families of \cite{ABBKP}, and were studied in \cite{RT}. The new key observation is that, under the mild condition $a + b \equiv 0 \ (\bmod \ 8)$, they are all BF-admissible. \par

Combining Theorems \ref{new-BF-non-vanishing}, \ref{main-A}, and \ref{main-B}, we conclude that vast families that consist of connected sums of $4$-manifolds with $b_1 > 0$ have non-trivial stable cohomotopy Seiberg-Witten invariants. The existence of such families of connected sums enables us to give several new application regarding the geometry and topology of smooth $4$-manifolds, which we present in the second part of our article. \par

It is known that connected sums of manifolds equipped with positive scalar curvature metrics admit such metrics as well \cite{G-L, S-Y}. Also known is that positive scalar curvature metric is stable under codimension $q \geq 3$ surgeries \cite{G-L, S-Y}. These results imply that the connected sums (\ref{irr-home-1}) and (\ref{irr-home-2}) admit positive scalar curvature metrics with respect to their standard smooth structures. Importantly, it means that stable cohomotopy Seiberg-Witten invariants of the connected sums of $4$-manifolds given in (\ref{irr-home-1}) and (\ref{irr-home-2}) above, equipped with standard smooth structures, vanish. This fact, together with Theorem \ref{new-BF-non-vanishing} and Theorem \ref{main-B}, allows us to prove the existence of pairwise homeomorphic but not diffeomorphic $4$-manifolds \textit{with trivial Seiberg-Witten invariants}. Namely, we get \textit{exotic} copies of standard $4$-manifolds which are connected sums of $\CP^2, \CPb, S^1 \x S^3, Y_p$, with trivial Seiberg-Witten invariants but non-trivial stable cohomotopy Seiberg-Witten invariants. 

\begin{cor}
For $i=1,2,3$, let $X_{i}$ be any one of the $4$-manifolds given in Theorem \ref{main-B}. Then any connected sum $\#^{j}_{i=1}{X}_{i}$ admits an exotic smooth structure, for $j=2,3$.   
\end{cor}

Moreover, by combining Theorem D in \cite{ishi-1} with Theorems \ref{main-A} and \ref{main-B} of our paper, we also obtain
\begin{cor}
Let $X$ be any closed, simply connected, non-spin, symplectic $4$-manifold with $b^+ \equiv 3 \ (\bmod \ 4)$. For $i=1, 2$, let $X_i$ be any one of the $4$-manifolds given in Theorem \ref{main-A} or Theorem \ref{main-B}. Then any connected sum $X \# \Big(\#_{i=1}^j X_i \Big)$ admits an exotic smooth structure, for $j=1, 2$.  
\end{cor}
Examples of closed non-spin and simply-connected $4$-manifolds (which necessarily satisfy $b^+ \equiv 3 \ (\bmod \ 4)$) can be pulled out from the large collections of \cite{ABBKP}, or from earlier works of various authors in this direction. (See for instance Gompf's pioneer work \cite{gom}.) \par 

Another main application we will give regards the Ricci flow solutions on smooth $4$-manifolds, and is discussed in Section $3$. This is tightly related to Conjecture 1.8 of Fang, Zhang and Zhang in \cite{fz-1}, as we will explain below. Let $X$ be a closed oriented Riemannian manifold of dimension $n \geq 3$. {The Ricci flow} on $X$ is the following evolution equation:
\begin{eqnarray*}
 \frac{\partial }{\partial t}{g}=-2{Ric}_{g}, 
\end{eqnarray*}
where ${Ric}_{g}$ is the Ricci curvature of the evolving Riemannian metric $g$. The Ricci flow was introduced in the celebrated work \cite{ha-0} of Hamilton in order to produce constant positive sectional curvature metrics on $3$-manifolds. Since the above equation does not preserve the volume in general, one often considers the \textit{normalized Ricci flow} on $X$:
\begin{eqnarray*}\label{Ricci}
 \frac{\partial }{\partial t}{g}=-2{Ric}_{g} + \frac{2}{n}\overline{s}_{g} {g}, 
\end{eqnarray*}
where $\overline{s}_{g}:={{\int}_{X} {s}_{g} d{\mu}_{g}}/{vol_{{g}}}$ and ${s}_{g}$ denotes the scalar curvature of the evolving Riemannian metric $g$, $vol_{g}:={\int}_{X}d{\mu}_{g}$ and $d{\mu}_{g}$ is the volume measure with respect to $g$. A one-parameter family of metric $\{g(t)\}$, where $t \in [0, T)$ for some $0<T\leq \infty$, is called a solution to the normalized Ricci flow if this satisfies the above equation at all $x \in X$ and $t \in [0, T)$. It is known that the normalized flow is equivalent to the unnormalized flow by reparametrizing in time $t$ and scaling the metric in space by a function of $t$. The volume of the solution metric to the normalized Ricci flow is constant in time. \par
Recall that a solution $\{g(t)\}$ to the normalized Ricci flow on a time interval $[0, T)$ is said to be maximal if it cannot be extended past time $T$. Let us also recall the following definition introduced by Hamilton \cite{ha-1, c-c}:
\begin{defn}\label{non-sin}
A maximal solution $\{g(t)\}$, $t \in [0, T)$ of the normalized Ricci flow on $X$ is called non-singular if $T=\infty$ and if the Riemannian curvature tensor $Rm_{g(t)}$ of $g(t)$ satisfies 
$$
\sup_{X \times [0, T)}|Rm_{g(t)}| < \infty. 
$$
\end{defn}
In his pioneer work, Hamilton \cite{ha-0} proved that, in dimension $3$, there exists a unique non-singular solution to the normalized Ricci flow if the initial metric is positive Ricci curvature. Moreover, in \cite{ha-1}, Hamilton classified non-singular solutions to the normalized Ricci flow on $3$-manifolds. This work played an important role in understanding long-time behavior of solutions of the Ricci flow on $3$-manifolds. On the other hand, Hamilton also proved that on any closed oriented Riemannian $4$-manifold with constant positive curvature operator, there is a unique non-singular solution to the normalized flow which converges to a smooth Riemannian metric of positive sectional curvature \cite{ha-2}. In \cite{fz-1}, Fang, Zhang and Zhang also studied the properties of non-singular solutions to the normalized Ricci flow in higher dimensions. Inspired by their work, the second author \cite{ishi} of the current article introduced the following definition:
\begin{defn}[\cite{ishi}]\label{bs}
A maximal solution $\{g(t)\}$, $t \in [0, T)$, to the normalized Ricci flow  on $X$ is called quasi-non-singular if $T=\infty$ and if the scalar curvature $s_{g(t)}$ of $g(t)$ satisfies 
$$
\sup_{X \times [0, T)}|{s}_{g(t)}| < \infty. 
$$
\end{defn}
Any non-singular solution is quasi-non-singular. In dimension $4$, it was observed in \cite{fz-1} that the existence of the non-singular solution of the normalized Ricci flow brings constraints on the topology of the $4$-manifold, and in particular on its Euler characteristic and signature. Based on this fact, the authors proposed a conjecture. To state their conjecture precisely, we need to recall the definition of Perelman's $\bar{\lambda}$ invariant \cite{p-1,p-2}. Let $g$ be any Riemannian metric on a closed oriented smooth manifold $X$ with dimension $n \geq 3$. Consider the least eigenvalue $\lambda_g$ of the elliptic operator $4 \Delta_g+s_g$, where $s_g$ denotes the scalar curvature of $g$, and $\Delta = d^*d= - \nabla\cdot\nabla $ is the positive-spectrum  Laplace-Beltrami operator associated with  $g$. $\lambda_g$ can be expressed in terms of Raleigh quotients as 
\begin{eqnarray*}
\lambda_g = \inf_{u} \frac{\int_X \left[ s_gu^2 + 4 |\nabla u|^2 \right] d\mu}{\int_M u^2d\mu}, 
\end{eqnarray*}
where the infimum is taken over all smooth, real-valued functions $u$ on $X$. Consider the the scale-invariant quantity $\lambda_g(vol_g)^{2/n}$, where $vol_g= \int_Md\mu_g$ denotes the total volume of $(X,g)$. By taking the supremum of this quantity over the space of all Riemannian metrics, we can define Perelman's $\bar{\lambda}$ invariant associated to $X$:
\begin{eqnarray}\label{pre}
\bar{\lambda}(X)= \sup_g \lambda_g(vol_g)^{2/n}.
\end{eqnarray}
The Fang-Zhang-Zhang conjecture can be stated as follows: 
\begin{conj}[Conjecture 1.8 in \cite{fz-1}]\label{conj-1}
Let $X$ be a closed oriented smooth Riemannian $4$-manifold with $||X|| \not=0$ and $\bar{\lambda}(X) < 0$, where $||X||$ denotes Gromov's simplicial volume. Suppose that there is a quasi-non-singular solution to the normalized Ricci flow on $X$. Then the following holds:
\begin{eqnarray}\label{gh-T}
2\eu(X) - 3|\sign(X)| \geq \frac{1}{1295{\pi}^2}||X||. 
\end{eqnarray}
 \end{conj}
In this article, we refer to this conjecture as the \textit{FZZ conjecture} in short. To the best of our knowledge, the FZZ conjecture remains open. In connection with this conjecture, the following problem arises naturally:  
\begin{pro}\label{prob-1}
Let $X$ be a closed oriented smooth $4$-manifold with $||X|| \not=0$, $\bar{\lambda}(X) < 0$ and satisfying the inequality (\ref{gh-T}). Then, is there always a quasi-non-singular solution to the normalized Ricci flow on $X$?
\end{pro}
This is nothing but the converse of Conjecture \ref{conj-1}. \par
We would like to introduce: 
\begin{defn}\label{def-property}
Let $X$ be a closed oriented topological $4$-manifold. We say $X$ has property $\mathcal R$ if $X$ satisfies the following properties. 
\begin{enumerate}
\item $X$ has  $||X|| \not=0$ and satisfies the strict case of the inequality (\ref{gh-T}):
\begin{eqnarray*}
2\eu(X) - 3|\sign(X)| > \frac{1}{1295{\pi}^2}||X||. 
\end{eqnarray*}
\item $X$ admits at least one smooth structure for which Perelman's $\bar{\lambda}$ invariant is negative and there is no quasi-non-singular solution to the normalized Ricci flow for any initial metric. 
\end{enumerate}
Similarly, we say $X$ has \textit{$\mathcal{R}-\infty$-property} if $X$ satisfies the above condition 1 and moreover admits infinitely many smooth structures for which Perelman's $\bar{\lambda}$ invariants are negative and there is no quasi-non-singular solution to the normalized Ricci flow for any initial metric. 
\end{defn}

We shall prove the following existence theorem of $4$-manifolds with property $\mathcal R$ in the sense of Definition \ref{def-property}. 
\begin{main}\label{main-AcC}
Let $X_{m}$ be a BF-admissible closed oriented smooth $4$-manifold and consider the following connected sum:
\begin{eqnarray*}
M^{{\ell}_{1}, {\ell}_{2}}_{g, h, j}:=(\#^{j}_{m=1}X_{m}) \# (\Sigma_{h} \times \Sigma_{g}) \# \ell_{1}({S}^{1} \times {S}^{3}) \# \ell_{2} \overline{{\mathbb C}{P}^{2}}, 
\end{eqnarray*}
where $j = 1, 2$, ${\ell}_{1}, {\ell}_{2} \geq 1$ and $g, h \geq 3$ are odd integers. Then, there are infinitely many sufficiently large integers $g, h, {\ell}_{1}, \ell_{2}$ for which $M^{{\ell}_{1}, {\ell}_{2}}_{g, h, j}$ has property $\mathcal R$. 
\end{main}
These integers $g, h, {\ell}_{1}, \ell_{2}$ depend on the topology of the connected sum $\#^{j}_{m=1}X_{m}$ in general. Since there are infinitely many BF-admissible closed $4$-manifolds by Theorems \ref{main-A} and \ref{main-B}, as a corollary to Theorem \ref{main-AcC}, we see that:
\begin{cor}
The converse of the FZZ conjecture fails to hold for vast families of $4$-manifolds. 
\end{cor}

We are also able to prove a similar existence theorem of $4$-manifolds with $\mathcal{R}-\infty$-property in the sense of Definition \ref{def-property} as follows:
\begin{main}\label{main-AcCC}
Let $X$ be a BF-admissible closed oriented smooth $4$-manifold and consider the following connected sum:
\begin{eqnarray*}
M^{{\ell}_{1}, {\ell}_{2}}_{g, h}:=X \# K3 \# (\Sigma_{h} \times \Sigma_{g}) \# \ell_{1}({S}^{1} \times {S}^{3}) \# \ell_{2} \overline{{\mathbb C}{P}^{2}}, 
\end{eqnarray*}
where $j = 1, 2$, ${\ell}_{1}, {\ell}_{2} \geq 1$ and $g, h \geq 3$ are odd integers. Then, there are infinitely many sufficiently large integers $g, h, {\ell}_{1}, \ell_{2}$ for which $M^{{\ell}_{1}, {\ell}_{2}}_{g, h}$ has $\mathcal{ R}-\infty$-property. 
\end{main}

In Subsection \ref{sub-54}, we will propose a stronger version of the the Conjecture \ref{conj-1}; see Conjecture \ref{conj-2} stated there. We shall moreover derive results analagous to Theorem \ref{main-AcC} and Theorem \ref{main-AcCC}; see Theorems \ref{main-EE-mu} and \ref{main-EEE-mu}.

\vspace{0.4in}
\noindent \textit{Acknowledgments.} The first author was partially supported by the NSF grant DMS-0906912. The second author is partially supported by the Grant-in-Aid for Scientific Research (C), Japan Society for the Promotion of Science, No. 20540090.

\newpage
%%%%%%%%%%%%%%%%%%%%%%%%%%%%%%%%%%%%%%%%%%%%%%%%%%%%%%%%%%%%%%%%%%%%%%%%%%
\section{Families of $4$-manifolds satisfying BF-axioms}
%%%%%%%%%%%%%%%%%%%%%%%%%%%%%%%%%%%%%%%%%%%%%%%%%%%%%%%%%%%%%%%%%%%%%%%%%%

In this section, we will be proving Theorems \ref{main-A} and \ref{main-B}, which were stated in the Introduction. 

%%%%%%%%%%%%%%%%%%%%%%%%%%%%%%%%%%%%%%%%%%%%%%%%%%%%%%%%%%%%%%%%%%%%%%%%%%
\subsection{Logarithmic transforms and Luttinger surgeries} \label{luttinger}
%%%%%%%%%%%%%%%%%%%%%%%%%%%%%%%%%%%%%%%%%%%%%%%%%%%%%%%%%%%%%%%%%%%%%%%%%%

Let $L$ be an embedded self-intersection zero $2$-torus in a $4$-manifold $X$ with oriented tubular neighborhood $N(L)$. A \emph{framing} of $N(L)$ is a choice of an orientation-preserving diffeomorphism $\xi :N(L) \to D^2 \x T^2$, giving an identification 
\begin{equation} \label{bettilog}
H_1(\partial (X \setminus N(L))) \cong H_1(L)\oplus \Z,
\end{equation}
where the last summand is generated by a positively oriented meridian $\mu_L$ of $L$. We can construct a new $4$-manifold $X' = X \setminus N(T) \cup_{\phi} D^2 \x T^2$ using a diffeomorphism $\phi: \partial (T^2 \x D^2) \to \partial N(L)$. This diffeomorphism is uniquely determined up to isotopy by the homology class 
\begin{equation*} 
\phi_*[\partial D^2] = p [\mu_L] + q [S^1_{\lambda}] \, ,
\end{equation*} 
where $S^1_{\lambda}$ is a \emph{push-off} of a primitive curve $\lambda$ in $L$ by the chosen framing $\xi$. To sum up, the result of the surgery is determined by the torus $L$, the framing $\xi$, the \emph{surgery curve} $\lambda$ and \emph{the surgery coefficient} $p/q \in \Q \cup \{\infty \}$. This data is encoded in the notation $X(L, \lambda, p/q)$ whenever the framing is clear from the context. The operation producing $X'= X(L, \lambda, p/q)$ is called the \emph{(generalized) logarithmic $p/q$ transform of $X$ along $L$} ---with surgery curve $\lambda$ and framing $\xi$, which we will denote by $(L, \lambda, p/q)$.

If $(X, \omega_X)$ is a symplectic manifold and $L$ is a Lagrangian torus in $X$, then $L$ admits a \emph{Weinstein neighborhood} $N(L)$, which is a tubular neighborhood of $L$ equipped with a canonical framing. This framing, called the \emph{Lagrangian framing} here, is characterized by the unique property that $x \x T^2 $, for any $x \in D^2$, corresponds to a Lagrangian submanifold of $X$ under it. Let $\xi$ be the Lagrangian framing and $S^1_{\lambda}$ be the \emph{Lagrangian push-off} of $\lambda$, i.e the push-off of $\lambda$ in this framing. The $(L, \lambda, 1/q)$ surgery with these choices can be performed \emph{symplectically}, providing us with ---a deformation class of--- a symplectic form $\omega_{X'}$ on $X'= X(L, \lambda, p/q)$ that agrees with $\omega_X$ on the complement of $N(L)$ \cite{ADK}. This special logarithmic transform is referred as \emph{Luttinger surgery}.

The classical topological invariants of $4$-manifolds we are interested in this article change under logarithmic transforms (and in particular under Luttinger surgeries) as follows: Euler characteristic and signature of $X'$ and $X$ are the same, yet their spin types may differ depending on the choice of $L$ and the surgery. It follows that when $\mu_L$ is nullhomologous and $S^1_{\lambda}$ is homologically essential in $X \setminus N(L)$, we have $b_1(X') = b_1(X) -1$ and $b_2(X')=b_2(X)-2$. On the other hand, when both $S^{1}_{\lambda}$ and $L$ are nullhomologous in $X \setminus N(L)$, 
\begin{equation*}
H_1(X(L,\lambda,p/q );\Z )= H_1(X;\Z ) \oplus \Z / p\Z \, . 
\end{equation*}
Lastly, applying the Seifert-Van Kampen theorem, we get: 
\begin{equation} \label{LogSVK}
\pi_1(X(L,\lambda,p/q)) = \pi_1(X \setminus N(T))/
\langle [\mu_{L}]^p [S^1_{\lambda}]^q =1 \rangle .
\end{equation}

It follows from the very definition that a logarithmic transform operation can be reversed, by performing a logarithmic transform along the \textit{core torus} of the surgery that now lies in $X'= X(L, \lambda, p/q)$ by an appropriate choice of the surgery curve and the surgery coefficient. It is an easy exercise to see that, the same holds true in the symplectic setting; i.e. a Luttinger surgery can be reversed to obtain back the original symplectic $4$-manifold. In this case, we will call it \textit{undoing} the corresponding logarithmic transform or Luttinger surgery.

In what follows, we will mainly be interested in Luttinger surgeries so as to conclude that the resulting $4$-manifolds we obtain satisfy the first assumption of Definition \ref{BF-adm-def}. Namely, we will be using the canonical class $\Gamma_{X}$ associated to the resulting symplectic form, so that 
\[SW_{X}(\Gamma_{X}) \equiv 1 \ (\bmod \ 2) \, , \text{and} \ c^{2}_{1}({\cal L}_{\Gamma_{X}}) = 2 \eu(X) + 3 \sign(X) .\]
The rest of the assumptions will be seen to be satisfied merely by looking at the topological effect of the underlying logarithmic transforms.

%%%%%%%%%%%%%%%%%%%%%%%%%%%%%%%%%%%%%%%%%%%%%%%%%%%%%%%%%%%%%%%%%%%%%%%%%%
\subsection{Surgered product manifolds}\label{sec-2}
%%%%%%%%%%%%%%%%%%%%%%%%%%%%%%%%%%%%%%%%%%%%%%%%%%%%%%%%%%%%%%%%%%%%%%%%%%

Let $X_0$ be the product of two Riemann surfaces $\Sigma_g$ and $\Sigma_h$, equipped with the product symplectic form. The second homology group $H_2(X_0)$ is generated by the homology classes of $\Sigma_g$, $\Sigma_h$ and the Lagrangian tori $a_i \x c_j$, $a_i \x d_j$, $b_i \x c_j$, $b_i \x d_j$, $i=1, \ldots, g$ and $j=1, \ldots, h$, where $a_i, b_i$ and $c_j, d_j$ are the symplectic pairs of homology generators of the surfaces $\Sigma_g$ and $\Sigma_h$, respectively. Assume that $X_1$ is obtained from $X_0$ via Luttinger surgeries along some of these homologically essential Lagrangian tori in $X_0$, such that: Each surgery is performed with surgery curve equal to one of $a_i, b_i, c_j, d_j$ carried on the torus and with surgery coefficient equal to $1/n$ with respect to the Lagrangian framing, for some $n \in \Z$. In the present article, we shall call these new symplectic manifolds \emph{surgered product manifolds}. Then we have 

\begin{lem}\label{surgered-prop-1}
All surgered product manifolds obtained from $\Sigma_{g} \times \Sigma_{h}$ with $b_{1} > 0$, are BF-admissible, for $g, h$ are positive odd integers.   
\end{lem}
\begin{proof}
Assume $g$ and $h$ are both odd positive integers. Since 
\[ b^+(\Sigma_g \x \Sigma_h) = 1 + 2 gh \, , \ \text{and} \ b_1(\Sigma_g \x \Sigma_h) =2(g + h) ,\]
the difference $b^+-b_1 \equiv 1 + 2(gh - g-h) \equiv 3 \pmod{4}$. If we perform a torus surgery along any one of the product Lagrangian tori with the surgery curve equals any one of the homology generators (namely $a_i$, $b_i$, $c_j$ or $d_j$ for $i= 1, \ldots g$, $j= 1, \ldots, h$) and the surgery coefficient equals $1/n$ with respect to the Lagrangian framing, then $b_1$ drops by one, as seen from the equation (\ref{bettilog}).
Note that we can just compute $b_1$ in $\Q$-coefficients so $n$ can be any integer here. \par

Since the torus surgery does not change the Euler characteristic, $b_2$ of the new manifold we get drops by two. Moreover, what dies in the new homology is nothing but the homology class of the torus we performed the surgery along as well as the homology class of the torus dual to it. (A detailed analysis of this fact can be found in \cite{hl}.) These Lagrangian tori made up a hyperbolic pair in the second homology of the original manifold, so we see that each one of $b^+$ and $b^-$ drop by one. Hence the difference $b^+ - b_1$ remains the same and equals to $3 \pmod{4}$ after the surgery, satisfying the second condition in Definition \ref{BF-adm-def}. \par 

Now, let $(X, \omega_X)$ be the resulting symplectic $4$-manifold obtained by a sequence of Luttinger surgeries of this sort in $X_0$. To guarantee that \linebreak $b^+(X) > 1$, one just should not get taken by the heat of this process and kill all pairs of Lagrangian tori. It suffices to leave one such pair; $\Sigma_g \x \{pt \}$ and $\{ pt \} \x \Sigma_h$ still descend to the new symplectic manifold as a hyperbolic pair, and together with another pair of Lagrangian tori we get $b^+(X_1) > 1$ as required. \par

For the class $\Gamma_{X}$ take any almost complex structure compatible with the symplectic form, on which Seiberg-Witten invariant evaluates as $1$ (and thus equals $1 \pmod{2}$) by Taubes' celebrated work in \cite{ta}. This particularly implies that the first condition in Definition \ref{BF-adm-def} is satisfied. \par 

The new set of generators for $H_1(X)$ is given by all $a_i$, $b_i$, $c_j$ or $d_j$ for $i= 1, \ldots g$, $j= 1, \ldots, h$ except for the used surgery curves. Now note that $a_p \x b_q$ and $c_r \x d_s$ for $p,q = 1, \ldots, g$ and $r,s = 1, \ldots, h$ (whichever still exist) are all trivial in $H_2(X)$. On the other hand all other possible products were prescribing Lagrangian tori in the symplectic manifold $X_0$. Since the canonical class of $X_0$ can be supported away from all these tori \cite{ADK}, these tori are still Lagrangian in $X$. Thus, if we choose $\mathcal{\Gamma}_X$ as an almost complex structure compatible with the symplectic form on $X$, then the evaluation $\mathfrak{S}^{ij}(\Gamma _X):=\frac{1}{2}< c_1(\Gamma _X) \cup \mathfrak{e}_i \cup \mathfrak{e}_j, [X] >$ is either trivially zero to begin with or is equal to evaluating $\omega$ on a Lagrangian torus, and thus, vanishes in all possible cases, satisfying the third condition in Definition \ref{BF-adm-def}.
\end{proof}

In addition to the manifolds of the type $\Sigma_g \x \Sigma_h$ for $g,h$ odd, it was observed in \cite{ishi-1} that a primary Kodaira surface is also BF-admissible. Both of these families of manifolds are indeed subfamilies of surgered manifolds, as we now show for the non-trivial case: \footnote{Tian-Jun Li has informed us that this observation was known to him. Also see \cite{hl}.}  

\begin{figure}[h]
\begin{center}
\includegraphics[scale=1.5]{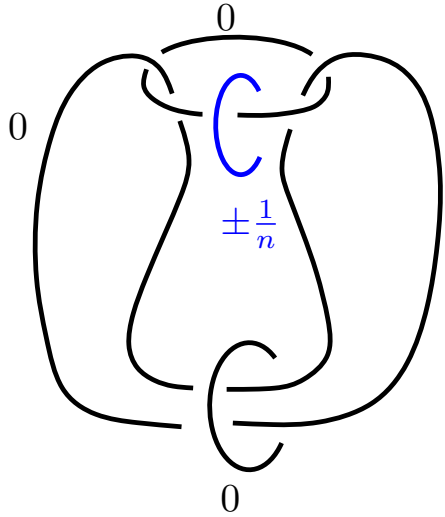}   
\caption{An $S^1$ invariant surgery diagram for a primary Kodaira surface.}
\end{center}
\end{figure}

\begin{lem}\label{surgered-prop-2}
A primary Kodaira surface $K$ is a surgered product manifold. In particular, $K$ is BF-admissible. 
\end{lem} 

\begin{proof}
Take $g=h=1$ and perform one Luttinger surgery along any one of the homologically essential tori listed above. Without loss of generality we can assume that this torus is $a \x c$ (where we drop the subindices as $g=h=1$). The resulting manifold can be described by the dimensionally reduced Kirby diagram given below. In the diagram one depicts $X_0=T^2 \x T^2$ as $S^1 \x T^3$ where the first $S^1$ component corresponds to $a$, and not drawn. Then since the diagram and the surgery are set in an $S^1$ invariant way, the Luttinger surgery amounts to performing a Dehn surgery along $c$ with coefficient $n$ in the $T^3$ component \cite{ABP}. The resulting diagram describes the smooth type of a primary Kodaira surface. 

Another way to see this is through the classification of Lagrangian torus bundles over tori (see \cite{geiges}). The projection onto $b \x d$ describes a Lagrangian torus bundle on $T^4$ equipped with the product symplectic form (i.e. the sum of the pullbacks of the volume forms on tori $a \x b$ and $c \x d$). The reader can verify that the Luttinger surgery along $a \x c$ in question yields a new Lagrangian torus bundle over a torus, where the fiber now is necessarily inessential in homology. As the result is a symplectic $4$-manifold admitting a Lagrangian torus bundle over a torus, it is a primary Kodaira surface. 
\end{proof}

\noindent Theorem \ref{main-A} now follows from Lemmas \ref{surgered-prop-1} and \ref{surgered-prop-2}.  

%%%%%%%%%%%%%%%%%%%%%%%%%%%%%%%%%%%%%%%%%%%%%%%%%%%%%%%%%%%%%%%%%%%%%%%%%%
\subsection{Families obtained from surgered product manifolds}\label{sec-3}
%%%%%%%%%%%%%%%%%%%%%%%%%%%%%%%%%%%%%%%%%%%%%%%%%%%%%%%%%%%%%%%%%%%%%%%%%%

We are now going to look at large families of $4$-manifolds constructed using solely the surgered products as building blocks. Such families, spanning a large portion of the geography plane were obtained in \cite{ABBKP}:

\begin{thm}[Theorem A in \cite{ABBKP}] \label{ABBKPthm}
Let $a$ and $b$ denote integers satisfying  $2a+3b \geq 0$, and $a+ b \equiv 0 \pmod{4}.$ If, in addition, $b \leq -2,$ then there exists a simply connected minimal symplectic $4$-manifold with Euler characteristic $a$ and signature $b$ and odd intersection form, except possibly for $(a,b)$ equal to $(7, -3)$, $(11, -3)$, $(13, -5)$, or $(15, -7)$.
\end{thm}

Note that the missing four lattice points given in the statement, using the minimal symplectic $\CP^2 \# 2 \CPb$ constructed by Akhmedov and Park, can be realized by the same methods of \cite{ABBKP}, as shown in \cite{A-P-inv}. For the lack of a better name, we will call all these manifolds as \textit{ABBKP manifolds} in short. A close look at these examples show that they are all obtained from surgered product manifolds via a couple of operations. Namely:

\begin{enumerate} 
\item[(1)] Symplectic blow-ups at points on the symplectic surfaces $\Sigma_g \x \{pt \}$ or $\{ pt \} \x \Sigma_h$ in the surgered product manifolds; and
\item[(2)] Symplectic fiber sums along \textit{symplectic} surfaces which are obtained from copies of $\Sigma_g \x \{pt\}$, $\{ pt \} \x \Sigma_h$ and exceptional spheres that might have been introduced during blow-ups.
\end{enumerate}  

Said differently, these manifolds are obtained by using symplectic building blocks $\Sigma_g \x \Sigma_h$ ---where $g$ and $h$ are \emph{not} necessarily odd, the above two operations, \emph{and} Luttinger surgeries with coefficients $\pm 1$ are performed along the product Lagrangian tori contained in them. This is because these Lagrangian tori are away from the standard symplectic surfaces $\Sigma_g \x \{pt \}$ or $\{ pt \} \x \Sigma_h$, and remain Lagrangian after blow-ups of fiber sums. Therefore, one can perform the above two operations and the Luttinger surgeries in any order to get the resulting symplectic $4$-manifold. 

To meet the first and second conditions in Definition \ref{BF-adm-def}, we only deal with those $X$ with $b^+(X)\equiv 3 \pmod{4}$. Since $b_1(X)=0$, the third condition in Definition \ref{BF-adm-def} is satisfied vacuously for these manifolds. Now, if we \textit{undo} any of the Luttinger surgeries, from our previous arguments in the proof of Lemma \ref{surgered-prop-1} we see that we re-introduce the hyperbolic pair of Lagrangian tori in the new resulting symplectic manifold, but all the conditions in Definition \ref{BF-adm-def} are still satisfied. Hence we see that:

\begin{thm} \label{undoingABBKP}
If one undoes any collection of the Luttinger surgeries involved in the construction of any one of the ABBKP manifold with $b^+ \equiv 3 \pmod{4}$, the resulting manifold meets all the conditions in Definition \ref{BF-adm-def}.
\end{thm}

\noindent Undoing these surgeries in simply-connected end products will re-introduce $b_1$ in a straightforward fashion. The change in fundamental group however is more subtle, and is to our interest mostly when we only undo one of the surgeries to get manifolds with fundamental group $\mathbb Z$ and perform the last surgery with general Luttinger surgery coefficient $\pm 1/p$ instead of $\pm 1$ to get ${\mathbb Z} \x {\mathbb Z}_m$, for which we can use homeomorphism criteria given by the following theorems:

\begin{thm}[Hambleton-Teichner \cite{h-t}, see also \cite{kru-Lee}.] \label{HTthm}
\ Let $X$ be a \linebreak smooth closed oriented $4$-manifold with infinite cyclic fundamental group. $X$ is classified up to homeomorphism by the fundamental group, the intersection on $H_{2}(X, {\mathbb Z})/Tors$ and the $w_{2}$-type. If in addition, $b_{2}(X) - |\sign(X)| \geq 6$, then $X$ is homeomorphic to the connected sum of $S^{1} \times S^{3}$ with a unique closed simply connected $4$-manifold.  In particular, $X$ is determined up to homeomorphism by its second Betti number $b_{2}(X)$, its signature $\tau(X)$ and its $w_{2}$-type. Particularly, $X$ is either spin or non-spin depending on the parity of its intersection form. 
\end{thm}

\begin{thm}[Hambleton-Kreck \cite{h-k}] \label{HKthm} 
Let $X$ be a closed smooth oriented $4$-manifold with finite cyclic fundamental group. Then $X$ is classified up to homeomorphism by the fundamental group, the intersection form on \linebreak $H_2(X; \Z)/ Tors$, and the $\omega_2$-type. Moreover, any isometry of the intersection form can be realized by a homeomorphism.
\end{thm}

A $0$-surgery along $\{pt\} \x S^1$ in $L(p,1) \x S^1$ yields a manifold with fundamental group ${\mathbb Z}_p$, which has the smallest homology among all other \linebreak $4$-manifolds of the same fundamental group, which we denote by $Y_p$. A bi-product of the above discussion gives rise to Theorem \ref{main-B}:

\begin{proof} \textbf{[Theorem \ref{main-B}]} \ 
In \cite{ABBKP}, a key ingredient in the constructions were the \textit{telescoping triples}. We recall the definition of a telescoping triple here for the convenience of reader: An ordered triple $(X, T_1, T_2)$ where $X$ is a symplectic $4$-manifold and $T_1$, $T_2$ are disjointly embedded Lagrangian tori is called a telescoping triple if
\begin{itemize}
\item[(i)] The tori $T_1$, $T_2$ span a $2$-dimensional subspace of $H_2(X; \R)$. \
\item[(ii)] $\pi_1(X) = \Z \oplus \Z$ and the inclusion induces an isomorphism  \linebreak $\pi_1(X \setminus (T_1 \cup T_2)) \to \pi_1(X)$, which in particular implies that the meridians of the $T_i$ are trivial in $\pi_1(X \setminus (T_1 \cup T_2))$. \
\item[(iii)] The image of the homomorphism induced by inclusion \linebreak $\pi_1(T) \to \pi_1(X)$ is a summand $\Z$ in $\pi_1(X)$. \
\item[(iv)] The homomorphism induced by inclusion $\pi_1(T_2) \to \pi_1(X)$ is an isomorphism. \
\end{itemize}

Each ABBKP manifold $X'$ is obtained using various telescoping triples. In particular, $X'$ can be viewed as obtained from a telescoping triple $(X, T_1, T_2)$ (say the `last' telescoping triple involved in the construction) after a $\pm 1$ Luttinger surger along $T_2$. The very properties of a telescoping triple implies that undoing the Luttinger surgery along the core-torus that descends from $T_2$ hands us back a symplectic $4$-manifold $Z$ with fundamental group $\Z$. Now if one performs a Luttinger surgery along $T_2$ in $Z$ with the same surgery curve but with surgery coefficient $1/p$ instead, from Seifert-Van Kampen calculation we get a symplectic $4$-manifold $Z_p$ with fundamental group $\Z_p$. (Note that the sign of the surgery does not effect the resulting fundamental group, so it is not relevant to our discussion here.) \par

We claim that the manifolds $Z$ and $Z_p$ constructed for each ABBKP manifold $X$ make up the families 
\[ \alpha {\mathbb C}{P}^{2} \# \beta \overline{{\mathbb C}{P}^{2}} \# ({S}^{1} \times {S}^{3}) \, \text{and} \]
\[ (\alpha -1) {\mathbb C}{P}^{2} \# (\beta -1) \overline{{\mathbb C}{P}^{2}} \# Y_{p} \, , \]
respectively. From Theorem \ref{ABBKPthm}, there is an $X$ with $a=\eu(X)$, $b=\sign(X)$ satisfying $a + b \equiv 0 \ (\bmod \ 8)$, and $b \leq -2$. (These constitute the `half' of the ABBKP manifolds, since we require $a + b \equiv 0 \ (\bmod \ 8)$ instead of $a + b \equiv 0 \ (\bmod \ 4)$.) Therefore, the Euler characteristic and signature of $Z$ and $Z_p$ are also equal to $a$ and $b$, respectively. Clearly, both are non-spin smooth $4$-manifolds, and satisfy $b_{2}(X) - |\sign(X)| \geq 6$. Now, $\pi_1(Z) = H_1(Z) = \Z$ and $\pi_1(Z_p)= H_1(Z_p) = \Z_p$ for $p$ odd, lands $Z$ in the same homeomorphism class of $(a+b)/2 {\mathbb C}{P}^{2} \# (a-b)/2 \overline{{\mathbb C}{P}^{2}} \# ({S}^{1} \times {S}^{3})$ by Theorem \ref{HTthm} and $Z_p$ in $ (a+b/2 -1) {\mathbb C}{P}^{2} \# (a-b/2 -1) \overline{{\mathbb C}{P}^{2}} \# Y_{p}$ by Theorem \ref{HKthm}, respectively. \par

To prove that the manifolds $Z$ and $Z_p$ are irreducible, we recall that they can equivalently be obtained from surgered products via two operations (1) and (2) discussed above. There are two key observations made in \cite{ABBKP} and \cite{A-P-inv} to conclude the minimality of ABBKP manifolds: First of all, the surgered products used in these constructions are minimal. After blow-ups minimality is lost in the pieces, however, the fiber sums that follow are performed along symplectic surfaces that intersect the new exceptional spheres in the way that Usher's theorem on minimality of symplectic fiber sums \cite{Usher} can be employed to conclude that the resulting symplectic $4$-manifold is minimal. The same observations hold true when one of the Luttinger surgeries goes undone, since the only difference now surfaces in one of the surgered products containing the corresponding Lagrangian torus being obtained from a product of Riemann surfaces with one less Luttinger surgery (and thus yielding a non-rational surface bundle over a non-rational surface, which has no $\pi_2$). Hence, both $Z$ and $Z_p$ are minimal symplectic $4$-manifolds with residually finite fundamental groups. By \cite{h-ko}, they are irreducible. \par

Lastly, our claim that manifolds $Z$ and $Z_p$ satisfy the BF-axioms, follow from Theorem \ref{undoingABBKP}.
\end{proof}

\begin{rmk}
In \cite{ABBKP} many more possible lattice points in the geography plane for $\sign \leq 4$ were realized by minimal symplectic $4$-manifolds, leaving out about 280 lattice points. Moreover, the small manifolds constructed by Akhmedov and Park in \cite{A-P-inv} leads to a slight enlargement of this region spanned by the minimal symplectic $4$-manifolds. These manifolds can also be used to enlarge our families obtained in Theorem \ref{main-B}---a similar discussion can be found in \cite{RT}. Nevertheless, we are content with the vast families we have got for the applications that will follow in the next chapter, and therefore will not discuss these slight extensions here.  
\end{rmk}

%%%%%%%%%%%%%%%%%%%%%%%%%%%%%%%%%%%%%%%%%%%%%%%%%%%%%%%%%%%%%%%%%%%%%%%%%%
\section{Normalized Ricci flow, simplicial volume, Einstein metrics, and the FZZ conjecture}\label{sec-5}
%%%%%%%%%%%%%%%%%%%%%%%%%%%%%%%%%%%%%%%%%%%%%%%%%%%%%%%%%%%%%%%%%%%%%%%%%%

The main purpose of this section to prove Theorems \ref{main-AcC} and \ref{main-AcCC}, which were stated in the Introduction. 

%%%%%%%%%%%%%%%%%%%%%%%%%%%%%%%%%%%%%%%%%%%%%%%%%%%%%%%%%%%%%%%%%%%%%%%%%%
\subsection{Curvature bounds arising from the Seiberg-Witten equations}\label{sub-50}
%%%%%%%%%%%%%%%%%%%%%%%%%%%%%%%%%%%%%%%%%%%%%%%%%%%%%%%%%%%%%%%%%%%%%%%%%%

First of all, let us recall
\begin{defn}[\cite{kro, leb-11, ishi-leb-2, leb-17}]\label{ishi-leb-2-key}
Let $X$ be a closed oriented smooth $4$-manifold with $b^+(X) \geq 2$. An element $\frak{a} \in H^2(X, {\mathbb Z})$/torsion $\subset H^2(X, {\mathbb R})$ is called monopole class of $X$ if there exists a spin${}^c$ structure $\Gamma_{X}$ with 
\begin{eqnarray*}
{c}^{\mathbb R}_{1}({\cal L}_{\Gamma_{X}}) = \frak{a} 
\end{eqnarray*} 
which has the property that the corresponding Seiberg-Witten monopole equations have a solution for every Riemannian metric on $X$. Here ${c}^{\mathbb R}_{1}({\cal L}_{\Gamma_{X}})$ is the image of the first Chern class ${c}_{1}({\cal L}_{\Gamma_{X}})$ of the complex line bundle ${\cal L}_{\Gamma_{X}}$ in $H^2(X, {\mathbb R})$. We shall denote the set of all monopole classes on $X$ by ${\frak C}(X)$.
\end{defn} 

One of the crucial properties of the set ${\frak C}(X)$ is the finiteness of  ${\frak C}(X)$ as follows. 
\begin{prop}[\cite{leb-17, ishi-leb-2}]\label{mono-bounded}
Let $X$ be a closed oriented smooth $4$-manifold with $b^+(X) \geq 2$. Then ${\frak C}(X)$ is a finite set. 
\end{prop}

It is known \cite{ishi-leb-2} that the non-triviality of the stable cohomotopy Seiberg-Witten invariants implies the existence of monopole classes. It is also known that the existence of non-zero monopole classes on a closed $4$-manifold $X$ implies that $X$ can not admit any Riemannian metric of positive scalar curvature. For instance, see \cite{ishi-leb-2}. Hence, the existence of monopole classes tells us a differential geometric information of $X$. Moreover, LeBrun \cite{leb-11, leb-17} proved the existence of monopole classes implies several interesting curvature bounds, which have many powerful differential geometric applications.  By combining Theorem \ref{new-BF-non-vanishing} with the curvature bounds of LeBrun \cite{leb-11, leb-17}, we are able to obtain the following curvature bound. We shall use the bounds (\ref{weyl-leb-sca-1}) and (\ref{weyl-leb-sca-2}) to prove Theorems \ref{main-AcC} and \ref{main-AcCC}. 

\begin{thm}[Theorem 53 in \cite{ishi-1}] \label{mono-key-bounds}
For $m =1,2,3$, let $X_m$ be a BF-admissible $4$-manifold in the sense of Definition \ref{BF-adm-def} and set as $c^{2}_{1}(X_{m}) = 2 \eu(X_{m}) + 3 \sign(X_{m})$. Suppose that $N$ is a closed oriented smooth $4$-manifold with $b^{+}(N)=0$. Consider a connected sum 
$$
M:=\Big(\#^{n}_{m=1}{X}_{m} \Big) \# N,
$$ 
where $n=2,3$. Then any Riemannian metric $g$ on $M$ satisfies the following curvature estimates:
\begin{eqnarray}\label{weyl-leb-sca-1}
{\int}_{M}{{s}^2_{g}}d{\mu}_{g} \geq {32}{\pi}^{2}\sum^{n}_{m=1}{c}^2_{1}(X_{m}), 
\end{eqnarray} 
\begin{eqnarray}\label{weyl-leb-sca-2}
{\int}_{M}\Big({s}_{g}-\sqrt{6}|W^{+}_{g}|\Big)^2 d{\mu}_{g} \geq 72{\pi}^{2}\sum^{m}_{n=1}{c}^2_{1}(X_{m}),  
\end{eqnarray}
where $s_{g}$ and $W^{+}_{g}$ denote respectively the scalar curvature and self-dual Weyl curvature of $g$. 
\end{thm}

%%%%%%%%%%%%%%%%%%%%%%%%%%%%%%%%%%%%%%%%%%%%%%%%%%%%%%%%%%%%%%%%%%%%%%%%%%
\subsection{Asymptotic behavior of the Ricci curvature}\label{sub-51}
%%%%%%%%%%%%%%%%%%%%%%%%%%%%%%%%%%%%%%%%%%%%%%%%%%%%%%%%%%%%%%%%%%%%%%%%%%

To prove Theorems \ref{main-AcC} and \ref{main-AcCC}, we need to study asymptotic behavior of the Ricci curvature of the solution of the normalized Ricci flow under connected sum of $4$-manifolds. The main result of this subsection is Theorem \ref{key-prop-3} stated below. We shall use the bound (\ref{weyl-leb-sca-1}) above at the crucial part of the proof.\par
Inspired by works of Cao \cite{cao-X} and Li \cite{li}, one parameter family $\bar{\lambda}_{k}$ of smooth invariants, where  $k \in {\mathbb R}$, was introduced in \cite{ishi-1}. It is called $\bar{\lambda}_{k}$ invariant. $\bar{\lambda}_{k}$ invariant includes Perelman's $\bar{\lambda}$ invariant as a special case. Indeed, $\bar{\lambda}_{1}=\bar{\lambda}$ holds. Let us recall its definition. \par 
We shall start with recalling the following definition which is essentially due to Li \cite{li}. In fact, the following definition in the case where $k \geq 1$ is nothing but Definition 41 in \cite{li}. We also notice that the following definition was also appeared as the equality (8) in \cite {o-s-w}: 
\begin{defn}[\cite{li, o-s-w}]
Let $X$ be a closed oriented Riemannian manifold with dimension $\geq 3$. Then, we define the following variant ${\mathcal F}_{k} : {\cal R}_{X} \times C^{\infty}(X) \rightarrow {\mathbb R}$ of the Perelman's $\mathcal F$-functional: 
\begin{eqnarray}\label{li-pere}
{\mathcal F}_{k}(g, f):={\int}_{X}\Big(k{s}_{g}+|\nabla f|^2 \Big){e}^{-f} d{\mu}_{g},
\end{eqnarray}
where $k$ is a real number $k \in {\mathbb R}$. We shall call this ${\mathcal F}_{k}$-functional.
\end{defn}
Notice that ${\mathcal F}_{1}$-functional is nothing but Perelman's ${\mathcal F}$-functional. Li \cite{li} showed that all functionals ${\mathcal F}_{k}$ with $k \geq 1$ have the monotonicity properties under a certain coupled system of Ricci flow. \par
As was already mentioned in \cite{li, lott} essentially, for a given metric $g$ and $k \in {\mathbb R}$, there exists a unique minimizer of the ${\cal F}_{k}$-functional under the constraint ${\int}_{X}{e}^{-f} d\mu_{g} =1$. In fact, by using a direct method of the elliptic regularity theory, one can see that the following infimum is always attained: 
\begin{eqnarray*}
{{\lambda}}(g)_{k}:=\inf_{f} \ \{ {\cal F}_{k}(g, f) \ | \ {\int}_{X}{e}^{-f} d\mu_{g} =1 \}. 
\end{eqnarray*}
Notice that $\lambda(g)_k$ is nothing but the least eigenvalue of the elliptic operator $4 \Delta_g+ks_g$. It is then natural to introduce the following: 
\begin{defn}
For any real number $k \in {\mathbb R}$, the $\bar{\lambda}_{k}$ invariant of $X$ is defined to be 
\begin{eqnarray*}\label{p-inv}
\bar{\lambda}_{k}(X)= \sup_{g \in {\cal R}_{X}}\lambda(g)_{k} (vol_g)^{2/n}.  
\end{eqnarray*}
\end{defn}
This is a diffeomorphism invariants of a smooth manifold. And it is clear that $\bar{\lambda}_{1}=\bar{\lambda}$ holds.  Then, the following result holds:
\begin{lem}\label{mini-scal-lem}
Let $X$ be a closed oriented Riemannian manifold of dimension $n \geq 3$ and assume that there is a real number $k$ such that the $\overline{{\lambda}}_{k}$ invariant of $X$ is negative, i.e., $\overline{{\lambda}}_{k}(X)<0$. If there is a solution $\{g(t)\}$, $t \in [0, T)$, to the normalized Ricci flow, then 
\begin{eqnarray*}
\hat{s}_{g(t)}:=\min_{x \in X}{s}_{g(t)}(x) \leq \frac{\overline{{\lambda}}_{k}(X)}{k(vol_{g(0)})^{2/n}} < 0, 
\end{eqnarray*}
where we define as $\hat{s}_{g} := \min_{x \in X}{s}_{g}(x)$ for a given Riemannian metric $g$.
\end{lem}

\begin{proof}
Let $\{g(t)\}$ be any solution to the normalized Ricci flow on $X$. Notice that $\lambda_{g(t)}$ can be expressed in terms of Raleigh quotients as 
\begin{eqnarray*}
\lambda_{g(t)} &=& \inf_{u} \frac{\int_X \left[ k s_{g(t)}u^2 + 4 |\nabla u|^2 \right] d{\mu}_{g(t)}}{\int_X u^2d{\mu}_{g(t)}}, 
\end{eqnarray*}
where the infimum is  taken over all smooth real-valued functions $u$ on $X$. Therefore we have
\begin{eqnarray*}
\lambda_{g(t)} &=& \inf_{u} \frac{\int_X \left[ k s_{g(t)}u^2 + 4 |\nabla u|^2 \right] d{\mu}_{g(t)}}{\int_X u^2d{\mu}_{g(t)}} \\
&\geq& \inf_{u} \frac{\int_X \left[ k \hat{s}_{g(t)} u^2 + 4 |\nabla u|^2 \right] d{\mu}_{g(t)}}{\int_X u^2d{\mu}_{g(t)}} \\
&\geq& k \hat{s}_{g(t)} \Big( \inf_{u} \frac{\int_X  u^2  d{\mu}_{g(t)}}{\int_X u^2d{\mu}_{g(t)}} \Big) = k \hat{s}_{g(t)}. 
\end{eqnarray*}
Hence $\lambda_{g(t)} \geq k \hat{s}_{g(t)}$ holds. On the other hand, by the very definition of $\overline{{\lambda}}_{k}$ invariant, we have $\overline{{\lambda}}_{k}(X) \geq \lambda_{g(t)}(vol_{g(t)})^{2/n}$. We therefore get the following:
\begin{eqnarray*}
\overline{{\lambda}}_{k}(X) \geq k \hat{s}_{g(t)}(vol_{g(t)})^{2/n}
\end{eqnarray*}
Since the normalized Ricci flow preserves the volume of the solution, we have $vol_{g(t)}= vol_{g(0)}$. Hence, we obtain
\begin{eqnarray*}
\overline{{\lambda}}_{k}(X) \geq k \hat{s}_{g(t)}(vol_{g(0)})^{2/n}.
\end{eqnarray*}
Equivalently, we obtain the desired bound:
\begin{eqnarray*}
\hat{s}_{g(t)} \leq \frac{\overline{{\lambda}}_{k}(X)}{k(vol_{g(0)})^{2/n}} < 0.
\end{eqnarray*}
\end{proof}
On the other hand, as one of interesting differential geometric invariants, there exists a natural diffeomorphism invariant arising from a variational problem for the total scalar curvature of Riemannian metrics on a closed oriented Riemannian manifold $X$ of dimension $n\geq 3$. As was conjectured by Yamabe, and later proved by Trudinger, Aubin, and Schoen, every conformal class on a smooth compact manifold contains a Riemannian metric of constant scalar curvature. Hence, for each conformal class $[g]=\{ vg ~|~v: X\to {\Bbb R}^+\}$, we are able to consider an associated number $Y_{[g]}$, which is so called Yamabe constant of the conformal class $[g]$ and defined by 
\begin{eqnarray*}
Y_{[g]} = \inf_{h \in [g]}  \frac{\int_X 
s_{{h}}~d\mu_{{h}}}{\left(\int_X 
d\mu_{{h}}\right)^{\frac{n-2}{n}}}, 
\end{eqnarray*}
where $s_{h}$ is the scalar curvature of the metric $h$ and $d\mu_{{h}}$ is the volume form with respect to the metric $h$. The Trudinger-Aubin-Schoen theorem tells us that this number is actually realized as the constant scalar curvature of some unit volume metric in the conformal class $[g]$. Then, Kobayashi \cite{kob} and Schoen \cite{sch-2} independently introduced the following interesting invariant of $X$:
\begin{eqnarray*}
{\mathcal Y}(X) = \sup_{\mathcal{C}}Y_{[g]}, 
\end{eqnarray*}
where $\mathcal{C}$ is the set of all conformal classes on $X$. This is now commonly known as the Yamabe invariant of $X$. It is known that ${\mathcal Y}(X) \leq 0$ if and only if $X$ does not admit a metric of positive scalar curvature. \par
$\bar{\lambda}_{k}$ invariant is closely related to the Yamabe invariant. Indeed, the following result holds.  
\begin{prop}[\cite{ishi-1}]\label{lambda-k-inv}
Suppose that $X$ is a smooth closed $n$-manifold, $n \geq 3$. Then the following holds: 
$$\bar{\lambda}_{k}(X) = \begin{cases}
     k{\mathcal Y}(X) & \text{ if  } {\mathcal Y}(X) \leq 0 \text{ and } k \geq \frac{n-2}{n-1}, \\
     +\infty  & \text{ if  } {\mathcal Y}(X) >  0 \text{ and } k > 0.
\end{cases}
$$
\end{prop}
We can prove the following bound by using Theorem \ref{new-BF-non-vanishing} and Proposition \ref{lambda-k-inv}. 
\begin{thm}\label{main-CCC}
For $m =1,2,3$, let $X_m$ be a BF-admissible $4$-manifold in the sense of Definition \ref{BF-adm-def}. Let $N$ be a closed oriented smooth $4$-manifold with $b^{+}(N)=0$. And assume that $\sum^n_{m=1}c^2_{1}(X_{m}) > 0$, where $n=2,3$ and $c^2_{1}(X_{m}) = 2\eu(X_{m}) + 3 \sign(X_{m})$. Then, for $n=2,3$ and any real number $k \geq \frac{2}{3}$, $\bar{\lambda}_{k}$ invariant of a connected sum $M:=(\#^{n}_{m=1}{X}_{m}) \# N$ satisfies the following bound: 
\begin{eqnarray*}
\bar{\lambda}_{k}(M) \leq {-4k{\pi}}\sqrt{2\sum^n_{m=1}c^2_{1}(X_{m})} < 0. 
\end{eqnarray*}
 
\end{thm}

\begin{proof}
As was already mentioned in Subsection \ref{sub-50}, the existence of the non-zero monopole classes implies the non-existence of metric of positive scalar curvature. Since Theorem \ref{new-BF-non-vanishing} tells us that the existence of non-zero monopole classes of $M$, we conclude that $M$ cannot admit any metric of positive scalar curvature. In particular, a result of Kobayashi \cite{kob} and this fact imply 
\begin{eqnarray}\label{nega-yama}
{\mathcal Y}(M) \leq 0. 
\end{eqnarray}
Under this situation, it is also known that the Yamabe invariant of $M$ is given by the following formula \cite{leb-44, leb-3}: 
\begin{eqnarray*}
{\mathcal Y}(M) = - \inf_{g} \Big( {\int}_{M}s^{2}_{g} d {\mu}_{g} \Big)^{\frac{1}{2}}. 
\end{eqnarray*}
On the other hand, we have the bound (\ref{weyl-leb-sca-1}) which holds for any Riemannian metric $g$ on $M$. 
\begin{eqnarray*}
{\int}_{M} s^{2}_{g} d{\mu}_{g} \geq 32{\pi}^{2} \sum ^{n}_{m=1} c^2_{1}(X_{m}). 
\end{eqnarray*}
We therefore conclude that the Yamabe invariant of $M$ satisfies 
\begin{eqnarray*}
{\mathcal Y}(M) \leq -4{\pi}\sqrt{2  \sum ^{n}_{m=1}c^2_{1}(X_{m})} < 0, 
\end{eqnarray*}
On the other hand, by Proposition \ref{lambda-k-inv} and (\ref{nega-yama}), we have the following equality:
$$
\bar{\lambda}_{k}(X) =  k{\mathcal Y}(X)
$$
if $k \geq \frac{2}{3}$. Therefore, we have the desired bound as follows. 
\begin{eqnarray*}
\bar{\lambda}_{k}(X) \leq -4k{\pi}\sqrt{2  \sum ^{n}_{m=1} c^2_{1}(X_{m})} < 0. 
\end{eqnarray*}
\end{proof}

Lemma \ref{mini-scal-lem} and Theorem \ref{main-CCC} tell us that the following result holds.
\begin{thm}\label{scalar-bound-1}
For $m =1,2,3$, let $X_m$ be a BF-admissible $4$-manifold in the sense of Definition \ref{BF-adm-def}. Assume that $\sum^n_{m=1}c^2_{1}(X_{m})> 0$ is satisfied, where $n=2,3$ and $c^2_{1}(X_{m}) = 2\eu(X_{m}) + 3 \sign(X_{m})$. Let $N$ be a closed oriented smooth $4$-manifold with $b^{+}(N)=0$. If there is a solution $\{g(t)\}$, $t \in [0, T)$, to the normalized Ricci flow on a connected sum $M:=(\#^{n}_{m=1}{X}_{m}) \# N$, where $n=2,3$, then the solution must satisfy the following bound:
\begin{eqnarray}\label{const-1}
\hat{s}_{g(t)}:=\min_{x \in X}{s}_{g(t)}(x) \leq - \Big( \frac{4{\pi}}{(vol_{g(0)})^{1/2}} \sqrt{2\sum^n_{m=1}c^2_{1}(X_{m})} \Big) < 0. 
\end{eqnarray}
\end{thm}
Notice that the right hand side of the bound (\ref{const-1}) is a negative constant of independent of  both $x \in X$ and $t$. \par
We also need to recall the following result. 
\begin{lem}[\cite{fz-1, ishi-1}]\label{FZZ-prop}
Let $X$ be a closed oriented Riemannian 4- manifold and assume that there is a long time solution $\{g(t)\}$, $t \in [0, \infty)$, to the normalized Ricci flow. Assume moreover that the solution satisfies $|{s}_{g(t)}| \leq C$ and 
\begin{eqnarray}\label{mini-bound}
\hat{s}_{g(t)} \leq -c <0,
\end{eqnarray}
where the constants $C$ and $c$ is independent of both $x \in X$ and time $t \in [0, \infty)$. Then, the trace-free part $\stackrel{\circ}{r}_{g(t)}$ of the Ricci curvature satisfies 
\begin{eqnarray*}
{\int}^{\infty}_{0} {\int}_{X} |\stackrel{\circ}{r}_{g(t)}|^2 d{\mu}_{g(t)}dt < \infty. 
\end{eqnarray*}
In particular, when $\ell \rightarrow \infty$, 
\begin{eqnarray}\label{fzz-ricci-0}
{\int}^{\ell+1}_{\ell} {\int}_{X} |\stackrel{\circ}{r}_{g(t)}|^2 d{\mu}_{g(t)}dt \longrightarrow 0. 
\end{eqnarray}
\end{lem}
Theorem \ref{scalar-bound-1} and Lemma \ref{FZZ-prop} imply the following result: 
\begin{thm}\label{key-prop-3}
For $m =1,2,3$, let $X_m$ be a BF-admissible $4$-manifold in the sense of Definition \ref{BF-adm-def}. And assume that $\sum^n_{m=1}c^2_{1}(X_{m})> 0$ is satisfied, where $n=2,3$ and $c^2_{1}(X_{m}) = 2\eu(X_{m}) + 3 \sign(X_{m})$. Let $N$ be a closed oriented smooth $4$-manifold with $b^{+}(N)=0$. If there is a quasi-non-singular solution $\{g(t)\}$ to the normalized Ricci flow on a connected sum $M:=(\#^{n}_{m=1}{X}_{m}) \# N$, where $n=2,3$, then
\begin{eqnarray}\label{fzz-ricci-011}
{\int}^{\ell+1}_{\ell} {\int}_{\Tilde{X}} |\stackrel{\circ}{r}_{g(t)}|^2 d{\mu}_{g(t)}dt \longrightarrow 0
\end{eqnarray}
holds when $\ell \rightarrow +\infty$. 
\end{thm}

%%%%%%%%%%%%%%%%%%%%%%%%%%%%%%%%%%%%%%%%%%%%%%%%%%%%%%%%%%%%%%%%%%%%%
\subsection{Obstruction}\label{sub-52}
%%%%%%%%%%%%%%%%%%%%%%%%%%%%%%%%%%%%%%%%%%%%%%%%%%%%%%%%%%%%%%%%%%%%%%

By using curvature bound (\ref{weyl-leb-sca-2}) and Theorem \ref{key-prop-3}, we shall prove an obstruction to the existence of non-singular solution to the normalized Ricci flow in dimension four. See Theorem \ref{Ricci-non-sin} stated below. \par 
Let $X$ be a closed oriented Riemannian $4$-manifold. Then, the Chern-Gauss-Bonnet formula and the Hirzebruch signature formula tell us that the following formulas hold for any Riemannian metric $g$ on $X$:
\begin{eqnarray*}
\sign(X)=\frac{1}{12{\pi}^2}{\int}_{X}\Big(|W^+_{g}|^2-|W^-_{g}|^2 \Big) d{\mu}_{g}, \\
\eu(X) = \frac{1}{8{\pi}^2}{\int}_{X}\Big(\frac{{s}^2_{g}}{24}+|W^+_{g}|^2+|W^-_{g}|^2-\frac{|\stackrel{\circ}{r}_{g}|^2}{2}  \Big) d{\mu}_{g}, 
\end{eqnarray*}
where $W^+_{g}$ and $W^-_{g}$ denote respectively the self-dual and anti-self-dual Weyl curvature of the metric $g$ and $\stackrel{\circ}{r}_{g}$ is the trace-free part of the Ricci curvature of the metric $g$. And $s_{g}$ is again the scalar curvature of the metric $g$ and $d\mu_{{g}}$ is the volume form with respect to $g$. By these formulas, we get the following:
\begin{eqnarray}\label{4-im}
2\eu(X) + 3\sign(X) = \frac{1}{4{\pi}^2}{\int}_{X}\Big(2|W^{+}_{g}|^2+\frac{{s}^2_{g}}{24}-\frac{|\stackrel{\circ}{r}_{g}|^2}{2} \Big) d{\mu}_{g},  
\end{eqnarray}
Then, we are able to prove the following result, which is used to prove Theorems \ref{main-AcC} and \ref{main-AcCC}. 

\begin{main}\label{Ricci-non-sin}
Let $N$ be a closed oriented smooth $4$-manifold with $b^{+}(N)=0$. For $m =1,2,3$, For $m =1,2,3$, let $X_m$ be a BF-admissible $4$-manifold in the sense of Definition \ref{BF-adm-def}. Assume also that $\sum^n_{m=1}c^2_{1}(X_{m}) > 0$ is satisfied, where $n=2,3$ and $c^2_{1}(X_{m}) = 2\eu(X_{m}) + 3 \sign(X_{m})$. Then, on a connected sum 
$$
M:=(\#_{m=1}^{n}{X}_{m}) \# N,
$$
where $n=2,3$, there is no quasi-non-singular solution to the normalized Ricci flow for any initial metric if the following holds:
\begin{eqnarray}\label{asm}
4{n}-\Big(2\eu(N)+3\sign(N) \Big) >  \frac{1}{3}\sum_{m=1}^{n}c^2_{1}(X_{m}). 
\end{eqnarray}
\end{main}

\begin{proof}
First of all, by (\ref{weyl-leb-sca-2}), we obtain the following bound which holds for any Riemannian metric $g$ on $M$: 
\begin{eqnarray*}
{\int}_{M}\Big({s}_{g}-\sqrt{6}|W^{+}_{g}|\Big)^2 d{\mu}_{g} \geq 72{\pi}^{2}\sum^{m}_{n=1}{c}^2_{1}(X_{m}).  
\end{eqnarray*}
On the other hand, as was already noticed in \cite{leb-11, ishi-leb-2}, we also have the following inequality for any Riemannian metric $g$ on $M$ (cf. Proposition 3.1 in \cite{leb-11}): 
\begin{eqnarray*}
\int_{M}\Big(2|W^{+}_{g}|^{2} + \frac{s^{2}_{g}}{24} \Big)d{\mu}_{g} \geq \frac{1}{27}{\int}_{M}\Big({s}_{g}-\sqrt{6}|W^{+}_{g}| \Big)^2 d{\mu}_{g}. 
\end{eqnarray*}
By the above inequalities, we conclude that any Riemannian metric $g$ on $M$ must satisfy the following bound: 
\begin{eqnarray}\label{u-11}
\frac{1}{4\pi^{2}}\int_{M}\Big(2|W^{+}_{g}|^{2} + \frac{s^{2}_{g}}{24} \Big)d{\mu}_{g} \geq \frac{2}{3}\sum_{m=1}^{n} c^2_{1}(X_{m}). 
\end{eqnarray}
Suppose now that there is a quasi-non-singular solution $\{g(t)\}$ to the normalized Ricci flow on $M$. Then, we have the following bound by (\ref{u-11})
\begin{eqnarray}\label{u-1333}
\frac{1}{4\pi^{2}}\int_{M}\Big(2|W^{+}_{g(t)}|^{2} + \frac{s^{2}_{g(t)}}{24} \Big)d{\mu}_{g(t)} \geq \frac{2}{3}\sum_{m=1}^{n}c^2_{1}(X_{m}). 
\end{eqnarray}
On the other hand, we also have the following inequality from (\ref{4-im})
\begin{eqnarray*}
2\eu(M) + 3\sign(M) = \frac{1}{4{\pi}^2}{\int}_{M}\Big(2|W^{+}_{g(t)}|^2+\frac{{s}^2_{g(t)}}{24}-\frac{|\stackrel{\circ}{r}_{g(t)}|^2}{2} \Big) d{\mu}_{g(t)}.  
\end{eqnarray*}
By this formula and (\ref{fzz-ricci-011}), we are able to obtain 
\begin{eqnarray*}
2\eu(M) + 3\sign(M) &=& \lim_{\ell \longrightarrow \infty} {\int}^{\ell+1}_{\ell} \Big(2\eu(M) + 3\sign(M) \Big)dt \\
&=& \lim_{\ell \longrightarrow \infty} \frac{1}{4{\pi}^2}{\int}^{\ell+1}_{\ell} {\int}_{M}\Big(2|W^{+}_{g(t)}|^2+\frac{{s}^2_{g(t)}}{24}-\frac{|\stackrel{\circ}{r}_{g(t)}|^2}{2} \Big) d{\mu}_{g(t)}dt \\
&=& \lim_{\ell \longrightarrow \infty}\frac{1}{4{\pi}^2}{\int}^{\ell+1}_{\ell} {\int}_{M}\Big(2|W^{+}_{g(t)}|^2+\frac{{s}^2_{g(t)}}{24}\Big) d{\mu}_{g(t)}dt.  
\end{eqnarray*}
This and the bound (\ref{u-1333}) imply 
\begin{eqnarray*}
2\eu(M) + 3\sign(M) &=&  \lim_{\ell \longrightarrow \infty}\frac{1}{4{\pi}^2}{\int}^{\ell+1}_{\ell} {\int}_{M}\Big(2|W^{+}_{g(t)}|^2+\frac{{s}^2_{g(t)}}{24}\Big) d{\mu}_{g(t)}dt \\ 
 &\geq& \lim_{\ell \longrightarrow \infty}\frac{2}{3} {\int}^{\ell+1}_{\ell} \sum_{m=1}^{n} c^2_{1}(X_{m}) dt \\
 &=& \frac{2}{3} \sum_{m=1}^{n} c^2_{1}(X_{m}). 
\end{eqnarray*}
On the other hand, a direct computation  tells us that
\begin{eqnarray*}
2\eu(M) + 3\sign(M) = \sum_{m=1}^{n} c^2_{1}(X_{m}) - 4{n}+\Big(2\eu(N)+3\sign(N) \Big). 
\end{eqnarray*}
We therefore obtain 
\begin{eqnarray*}
\sum_{m=1}^{n} c^2_{1}(X_{m}) - 4{n}+\Big(2\eu(N)+3\sign(N) \Big) \geq \frac{2}{3} \sum_{m=1}^{n} c^2_{1}(X_{m}).
\end{eqnarray*}
Equivalently, 
\begin{eqnarray*}
4{n}-\Big(2\eu(N)+3\sign(N) \Big) \leq  \frac{1}{3} \sum_{m=1}^{n} c^2_{1}(X_{m}). 
\end{eqnarray*}
By contraposition, we are able to get the desired result. Namely, under (\ref{asm}),  there is no quasi-non-singular solution to the normalized Ricci flow on the connected sum $M$ for any initial metric. 
\end{proof}

As a special case of Theorem \ref{Ricci-non-sin}, we obtain
\begin{cor}\label{speical-ein}
For $m=1,2$, let $X_{m}$ be a BF-admissible $4$-manifold. Consider a connected sum 
$$
M:=(\#_{m=1}^{j} X_{m}) \# (\Sigma_{g} \times \Sigma_{h}) \# \ell_{1}({S}^{1} \times {S}^{3}) \# \ell_{2} \overline{{\mathbb C}{P}^{2}}, 
$$
 where $j=1,2$, $\ell_{1}, \ell_{2} \geq 0$ and $g, h$ are odd integers $\geq 1$. Then there is no quasi-non-singular solution to the normalized Ricci flow on $M$ if 
\begin{eqnarray*}
4(2+\ell_{1}) + \ell_{2} > \frac{1}{3}\Big( \sum_{m=1}^{j} 2\eu(X_{m})+3\sign(X_{m})+4(1-h)(1-g) \Big).  
\end{eqnarray*}
\end{cor}
\begin{proof}
Theorem \ref{main-A} particularly tells us that $\Sigma_{g} \times \Sigma_{h}$ is BF-admissible. Notice also that we have $2\eu(N)+3\sign(N) = 4-4{\ell}_{1} - {\ell}_{1}$ by setting as $N := \ell_{1}({S}^{1} \times {S}^{3}) \# \ell_{2} \overline{{\mathbb C}{P}^{2}}$. By taking $n=3$ in the inequality (\ref{asm}), we have the desired result. 
\end{proof}

%%%%%%%%%%%%%%%%%%%%%%%%%%%%%%%%%%%%%%%%%%%%%%%%%%%%%%%%%%%%%%%%%%%%%
\subsection{Proof of Theorem \ref{main-AcC}}\label{sub-53}
%%%%%%%%%%%%%%%%%%%%%%%%%%%%%%%%%%%%%%%%%%%%%%%%%%%%%%%%%%%%%%%%%%%%%%

In this section, we shall prove Theorem \ref{main-AcC} stated in Introduction. \par
For the reader, let us recall the definition of simplicial volume due to Gromov \cite{gromov}. Let $M$ be a closed manifold. We denote by ${C}_{*}(M):=\sum^{\infty}_{k=0}{C}_{k}(M)$ the real coefficient singular chain complex of $M$. A chain $c \in {C}_{k}(M)$ is a finite combination $\sum{r}_{i}{\sigma}_{i}$ of singular simplexes ${\sigma}_{i} : {\Delta}^k \rightarrow M$ with real coefficients ${r}_{i}$. We define the norm $|c|$ of $c$ by $|c| : = \sum|r_{i}| \geq 0$. If $[\eta] \in H_{*}(M, {\mathbb R})$ is any homology class, then the norm $||\eta||$ of $[\eta]$ is define as 
\begin{eqnarray*}
||\eta||:=\inf \{|\frak a| : [{\frak a}] \in H_{*}(M, {\mathbb R}), [\frak a]=[\eta]\}, 
\end{eqnarray*}
where the infimum is taken over all cycles representing $\eta$. 
Suppose that $M$ is moreover oriented. Then we have the fundamental class $[M] \in H_{n}(M, {\mathbb R})$ of $M$. We then define the {simplicial volume} of $M$ by $||M|| \in [0, \infty)$. It is known that any simply connected manifold $M$ satisfies $||M||=0$. \par
First of all, we need to prove the following lemma. 
\begin{lem}\label{simplicial-lem}
Let $X_{m}$ be a closed $4$-manifold and consider a connected sum:
\begin{eqnarray*}
M:=(\#^{j}_{m=1} X_{m}) \# k (\Sigma_{h} \times \Sigma_{g}) \# \ell_{1}({S}^{1} \times {S}^{3}) \# \ell_{2} \overline{{\mathbb C}{P}^{2}}, 
\end{eqnarray*}
where $g, h \geq 1$, $j, k \geq 1$ and $\ell_{1}, \ell_{2} \geq 0$. Then the simplicial volume of $M$ is given by
\begin{eqnarray}\label{simplical-conn}
||M|| = 24k(g-1)(h-1) + \sum^{j}_{m=1}||X_{m}||.    
\end{eqnarray}
On the other hand, we have
\begin{eqnarray*}
2\eu(M)+3\sign(M) &=& \Big(\sum^{j}_{m=1} 2\eu(X_{m})+3\sign(X_{m}) \Big)+4k(g-1)(h-1)\\
&-&4(j+k-1+\ell_{1})-{\ell}_{2}, \\
2\eu(M)-3\sign(M) &=& \Big( \sum^{j}_{m=1} 2\eu(X_{m})-3\sign(X_{m}) \Big)+4k(g-1)(h-1) \\
                  &-& 4(j+k-1+\ell_{1})+5{\ell}_{2}.
\end{eqnarray*}

\end{lem}

\begin{proof}
It is known that the simplicial volume of the connected sum satisfies the following formula \cite{gromov, be}:
\begin{eqnarray*}
||M_{1} \# M_{2}|| = ||M_{1}|| + ||M_{2}||
\end{eqnarray*}
Since it is known that $||{S}^{1} \times {S}^{3} ||=0$ and $||\overline{{\mathbb C}{P}^{2}}||=0$ hold, the above formula tells us that 
\begin{eqnarray*}
||M|| = k||\Sigma_{h} \times \Sigma_{g} || + \sum^{j}_{m=1}||X_{m}||.    
\end{eqnarray*}
On the other hand, the following result is proved in \cite{Bucher-Karlsson(2007)}:
\begin{eqnarray*}
||\Sigma_{h} \times \Sigma_{g}|| =  24(g-1)(h-1).  
\end{eqnarray*}
Therefore, we have the formula (\ref{simplical-conn}). One can also deduce the formulas on $2\eu(M)+3\sign(M)$ and $2\eu(M)-3\sign(M)$ by direct computations. 
\end{proof}

We also have
\begin{lem}\label{cosimplicial-lem-2}
Let $X_{m}$ be a closed oriented smooth $4$-manifold and consider the following connected sum:
\begin{eqnarray*}
M:=(\#^{j}_{m=1}X_{m}) \# (\Sigma_{h} \times \Sigma_{g}) \# \ell_{1}({S}^{1} \times {S}^{3}) \# \ell_{2} \overline{{\mathbb C}{P}^{2}}, 
\end{eqnarray*}
where $j = 1,2$. For any pair $(g,h)$ of positive integers $\geq 2$, define the following positive number: 
\begin{eqnarray}\label{kappa-1}
{\kappa}(g,h):=  4(1-h)(1-g) - \frac{24(1-h)(1-g)}{1295{\pi}^2} > 0. 
\end{eqnarray}
Then, there are infinitely many sufficiently large integers $g, h, {\ell}_{1}, {\ell}_{2}$ for which  the following three conditions are satisfied simultaneously: 
\begin{eqnarray}\label{connected-condition-2}
\sum^{j}_{m=1}\Big( 2\eu(X_{m}) - 3\sign(X_{m}) \Big)   >  - {\kappa}(g,h) + \frac{ ||X ||}{1295{\pi}^2} + 4(j + {\ell}_{1}) -5 \ell_{2},  
\end{eqnarray}
\begin{eqnarray}\label{connected-condition-3}
\sum^{j}_{m=1}\Big( 2\eu(X_{m}) + 3\sign(X_{m}) \Big) > - {\kappa}(g,h) + \frac{ ||X ||}{1295{\pi}^2} + 4(j + {\ell}_{1}) + \ell_{2},  
\end{eqnarray}
\begin{eqnarray}\label{connected-condition-4}
4(j+\ell_{1}) + \ell_{2} > \frac{1}{3}\Big(\sum^{j}_{m=1}\Big( 2\eu(X_{m})+ 3\sign(X_{m}) \Big) +4(1-h)(1-g) \Big),  
\end{eqnarray}
where set as $|| X || := \sum^{j}_{m}|| X_{m} || \in [0, \infty)$. 
\end{lem}
\begin{proof}
First of all, notice that the inequality (\ref{connected-condition-2}) is always satisfied by taking sufficiently large $\ell_{2}$ for any fixed ${\ell}_{1}, g, h$. On the other hand, the inequality (\ref{connected-condition-3}) is equivalent to 
\begin{eqnarray}\label{2-connected-condition-3}
\sum^{j}_{m=1}\Big( 2\eu(X_{m}) + 3\sign(X_{m}) \Big) + {\kappa}(g,h) - \frac{ ||X ||}{1295{\pi}^2} > 4(j + {\ell}_{1}) + \ell_{2}. 
\end{eqnarray}
Therefore, by (\ref{connected-condition-4}) and (\ref{2-connected-condition-3}), it is enough to prove that there exist infinitely many sufficiently large positive integers ${\ell}_{1}, {\ell}_{2}, g, h$ satisfying
\begin{eqnarray}\label{3-connected-condition-3}
 c_{j}+ {\kappa}(g,h) - \frac{ ||X ||}{1295{\pi}^2} > 4(j + {\ell}_{1}) + \ell_{2} > \frac{1}{3}\Big(c_{j} + 4(1-h)(1-g) \Big). 
\end{eqnarray}
where $c_{j}:= \sum^{j}_{m=1}\Big( 2\eu(X_{m}) + 3\sign(X_{m}) \Big)$. We set as$$
A := c_{j}+ {\kappa}(g,h) - \frac{ ||X ||}{1295{\pi}^2}, \ B:= \frac{1}{3}\Big(c_{j} + 4(1-h)(1-g) \Big), 
$$
namely, (\ref{3-connected-condition-3}) is nothing but $A > 4(2 + {\ell}_{1}) + \ell_{2} > B$. Notice that both $A$ and $B$ can become sufficiently large positive integers by taking sufficiently large $g$ or $h$. We also have 
\begin{eqnarray*}
A - B = \frac{2}{3}{c}_{j} + \Big( \frac{8}{3} - \frac{24}{1295{\pi}^2} \Big)(1-h)(1-g) - \frac{ ||X ||}{1295{\pi}^2}. 
\end{eqnarray*}
From this, we see that $A - B$ can become a large positive integer by taking large $g$ or $h$. Since there are infinitely many choices of such $g$ and $h$, we are able to conclude that there are also infinitely many $\ell_{1}, \ell_{2}$ satisfying $A > 4(2 + {\ell}_{1}) + \ell_{2} > B$. By taking sufficiently large $g$ or $h$, we are also able to find a sufficiently large ${\ell}_{2}$ satisfying the inequality (\ref{connected-condition-2}), where notice that ${\kappa}(g,h) > 0$ and also that we can take as $ 4(j + {\ell}_{1}) -5 \ell_{2} < 0$.  
\end{proof}

Lemma \ref{simplicial-lem} and Lemma \ref{cosimplicial-lem-2} imply 
\begin{prop}\label{prop-FZZ-HTin}
Let $X_{m}$ be a closed oriented smooth $4$-manifold and consider the a connected sum:
\begin{eqnarray*}
M:=(\#^{j}_{m=1}X_{m}) \# (\Sigma_{h} \times \Sigma_{g}) \# \ell_{1}({S}^{1} \times {S}^{3}) \# \ell_{2} \overline{{\mathbb C}{P}^{2}}, 
\end{eqnarray*}
where $j = 1, 2$. Then, there are infinitely many sufficiently large integers $g, h, {\ell}_{1}, \ell_{2}$ for which  the following two conditions are satisfies simultaneously: 
\begin{eqnarray}\label{Hit-Tho-1}
2\eu(M) - 3|\sign(X)| > \frac{|| M ||}{1295{\pi}^2},  
\end{eqnarray}
\begin{eqnarray}\label{Hit-Tho-2}
4(j+\ell_{1}) + \ell_{2} > \frac{1}{3}\Big(\sum^{j}_{m=1}\Big( 2\eu(X_{m})+ 3\sign(X_{m}) \Big) +4(1-h)(1-g) \Big).  
\end{eqnarray}
\end{prop}

\begin{proof}
Notice that (\ref{Hit-Tho-2}) is nothing but (\ref{connected-condition-4}). On the other hand, by Lemma \ref{simplicial-lem}, we have 
\begin{eqnarray*}
2\eu(M)+3\sign(M) &=& \Big(\sum^{j}_{m=1} 2\eu(X_{m})+3\sign(X_{m}) \Big)+4(g-1)(h-1)\\
&-&4(j+\ell_{1})-{\ell}_{2}. 
\end{eqnarray*}
By (\ref{simplical-conn}), we also obtain 
\begin{eqnarray*}
\frac{||M||}{1295{\pi}^{2}} = \frac{24}{1295{\pi}^{2}}(g-1)(h-1) + \frac{1}{1295{\pi}^{2}} \sum^{j}_{m=1}||X_{m}||.    
\end{eqnarray*}
Therefore, the inequality (\ref{connected-condition-3}) is nothing but 
\begin{eqnarray}\label{Hit-Tho-11}
2\eu(M) + 3\sign(X) > \frac{|| M ||}{1295{\pi}^2}. 
\end{eqnarray}
Similarly, since Lemma \ref{simplicial-lem} also tells us that 
\begin{eqnarray*}
2\eu(M)-3\sign(M) &=& \Big( \sum^{j}_{m=1} 2\eu(X_{m})-3\sign(X_{m}) \Big)+4(g-1)(h-1)\\
&-&4(j+\ell_{1})+5{\ell}_{2}, 
\end{eqnarray*}
the inequality (\ref{connected-condition-2}) is nothing but 
\begin{eqnarray}\label{Hit-Tho-12}
2\eu(M) - 3\sign(X) > \frac{|| M ||}{1295{\pi}^2}. 
\end{eqnarray}
By (\ref{Hit-Tho-11}) and (\ref{Hit-Tho-12}), we obtain (\ref{Hit-Tho-1}) as desired. 
\end{proof}

We also have 
\begin{cor}\label{cor-perel}
Let $X_{m}$ be a BF-admissible closed oriented smooth $4$-manifold and consider the a connected sum:
\begin{eqnarray*}
M:=(\#^{j}_{m=1}X_{m}) \# (\Sigma_{h} \times \Sigma_{g}) \# \ell_{1}({S}^{1} \times {S}^{3}) \# \ell_{2} \overline{{\mathbb C}{P}^{2}}, 
\end{eqnarray*}
where $j = 1, 2$, ${\ell}_{1}, {\ell}_{2} \geq 1$ and $g, h \geq 3$ are odd integers such that 
\begin{eqnarray}\label{negative-main-condition}
c^{j}_{g,h}:=\sum^{j}_{m=1}\Big( 2\eu(X_{m})+ 3\sign(X_{m}) \Big) +4(1-h)(1-g) > 0  
\end{eqnarray}
Then, for any real number $k \geq \frac{2}{3}$, $\bar{\lambda}_{k}$ invariant of the connected sum $M$ is given by 
\begin{eqnarray}\label{bound-c-p}
\bar{\lambda}_{k}(M) \leq {-4k{\pi}}\sqrt{2c^{j}_{g,h}} < 0.
\end{eqnarray}

\end{cor}
\begin{proof}
By Theorem \ref{main-A}, the product $\Sigma_{h} \times \Sigma_{g}$ of Riemann surface with odd genus is BF-admissible. Therefore, Theorem \ref{main-CCC} implies the desired bound (\ref{bound-c-p}). 
\end{proof}

We are now in a position to prove Theorem \ref{main-AcC} as follows. 
\begin{thm}\label{main-existence-thm}
Let $X_{m}$ be a BF-admissible closed oriented smooth $4$-manifold and consider the following connected sum:
\begin{eqnarray*}
M^{{\ell}_{1}, {\ell}_{2}}_{g,h,j}:=(\#^{j}_{m=1}X_{m}) \# (\Sigma_{h} \times \Sigma_{g}) \# \ell_{1}({S}^{1} \times {S}^{3}) \# \ell_{2} \overline{{\mathbb C}{P}^{2}}, 
\end{eqnarray*}
where $j = 1, 2$, ${\ell}_{1}, {\ell}_{2} \geq 1$ and $g, h \geq 3$ are odd integers. Then, there are infinitely many sufficiently large integers $g, h, {\ell}_{1}, \ell_{2}$ for which $M$ has property $\mathcal R$. 
\end{thm}
\begin{proof}
By (\ref{simplical-conn}), we have 
\begin{eqnarray*}
||M|| = 24(g-1)(h-1) + \sum^{j}_{m=1}||X_{m}||.    
\end{eqnarray*}
Therefore we have $||M|| \not=0$ for any $g,h > 1$. On the other hand, by Corollary \ref{cor-perel}, under $c^{j}_{g,h}> 0$, we have $\bar{\lambda}(M) < 0$. Notice that $c^{j}_{g,h}> 0$ is always satisfied for sufficiently large $g, h > 1$ by (\ref{negative-main-condition}). Moreover, Proposition \ref{prop-FZZ-HTin} tells us that there are infinitely many sufficiently large integers $g, h, {\ell}_{1}, \ell_{2}$ for which (\ref{Hit-Tho-1}) and (\ref{Hit-Tho-2}) are satisfied simultaneously. By (\ref{Hit-Tho-1}), $M$ satisfies the strict case of the inequality (\ref{gh-T}). On the other hand, under (\ref{Hit-Tho-2}), i.e., 
\begin{eqnarray*}
4(j+\ell_{1}) + \ell_{2} > \frac{1}{3}\Big(\sum^{j}_{m=1}\Big( 2\eu(X_{m})+ 3\sign(X_{m}) \Big) +4(1-h)(1-g) \Big),  
\end{eqnarray*}
there is no quasi-non-singular solution to the normalized Ricci flow on $M$ for any initial metric by Corollary \ref{speical-ein}. Hence, the desired result follows. 
\end{proof}

%%%%%%%%%%%%%%%%%%%%%%%%%%%%%%%%%%%%%%%%%%%%%%%%%%%%%%%%%%%%%%%%%%%%%
\subsection{Proof of Theorem \ref{main-AcCC}}\label{sub-533}
%%%%%%%%%%%%%%%%%%%%%%%%%%%%%%%%%%%%%%%%%%%%%%%%%%%%%%%%%%%%%%%%%%%%%%

Recall a construction of a certain sequence of homotopy $K3$ surfaces. See also \cite{BPV}. Let $Y_{0}$ be a Kummer surface with an elliptic fibration $Y_{0} \rightarrow {\mathbb C}{P}^{1}$. Let $Y_{\ell}$ be obtained from $Y_{0}$ by performing a logarithmic transformation of order $2 m + 1$ on a non-singular fiber of $Y_{0}$. Then, $Y_{m}$ are simply connected spin manifolds with $b^{+}(Y_{m}) = 3$ and $b^{-}(Y_{m}) = 19$. By the Freedman classification \cite{freedman}, $Y_{m}$ must be homeomorphic to a $K3$ surface. And $Y_{m}$ is a K{\"{a}}hler surface with $b^{+}(Y_{m}) > 1$ and hence a result of Witten \cite{w} tells us that $\pm {c}_{1}(Y_{m})$ are monopole classes of $Y_{m}$ for each $m$. Notice also that $Y_{m}$ is BF-admissible. By using $Y_{m}$, we are able to prove Theorem \ref{main-AcCC}:
\begin{thm}
Let $X$ be a BF-admissible closed oriented smooth $4$-manifold and consider the following connected sum:
\begin{eqnarray}\label{AcC-connected}
M^{{\ell}_{1}, {\ell}_{2}}_{g,h}:=X \# K3 \# (\Sigma_{h} \times \Sigma_{g}) \# \ell_{1}({S}^{1} \times {S}^{3}) \# \ell_{2} \overline{{\mathbb C}{P}^{2}}, 
\end{eqnarray}
where ${\ell}_{1}, {\ell}_{2} \geq 1$ and $g, h \geq 3$ are odd integers. Then, there are infinitely many sufficiently large integers $g, h, {\ell}_{1}, \ell_{2}$ for which $M^{{\ell}_{1}, {\ell}_{2}}_{g,h}$ has $\mathcal{R}-\infty$-property. 
\end{thm}

\begin{proof}
First of all, notice that $X$ has at least one monopole class $c_{1}(X)$ because $X$ is BF-admissible (see Definition \ref{BF-adm-def}) and hence $X$ has non-trivial stable cohomotopy Seiberg-Witten invariants. Then, consider the following connected sum which is homeomorphic to (\ref{AcC-connected}) for any $m$:
\begin{eqnarray*}
Z^{{\ell}_{1}, {\ell}_{2}}_{g, h} (m):= X \# Y_{m} \# (\Sigma_{h} \times \Sigma_{g}) \# \ell_{1}({S}^{1} \times {S}^{3}) \# \ell_{2} \overline{{\mathbb C}{P}^{2}}. 
\end{eqnarray*}
For each $g,h,{\ell}_{1}, {\ell}_{2}$, notice that the connected sum $Z^{{\ell}_{1}, {\ell}_{2}}_{g, h} (m)$ has non-trivial stable cohomotopy Seiberg-Witten invariants by Theorem \ref{new-BF-non-vanishing}. In particular, $X$ has monopole classes which are give by
\begin{eqnarray}\label{connected-f}
\pm {c}_{1}(X) \pm {c}_{1}(Y_{m}) + \sum^{b_{2}(N)}_{i=1} \pm{E}_{i},
\end{eqnarray}
where we set $N := \ell_{1}({S}^{1} \times {S}^{3}) \# \ell_{2} \overline{{\mathbb C}{P}^{2}}$ and $E_{1}, E_{2}, \cdots, E_{k}$ is a set of generators for $H^2(N, {\mathbb Z})$/torsion relative to which the intersection form is diagonal and the $\pm$ signs are arbitrary and independent of one another. \par
Then, for each $g,h,{\ell}_{1}, {\ell}_{2}$, we show that 
\begin{eqnarray}\label{infinite-diffeo}
\{Z^{{\ell}_{1}, {\ell}_{2}}_{g, h}(m) \}_{m \in {\mathbb N}}
\end{eqnarray}
contains infinitely many diffeo types. In fact, suppose that the sequence (\ref{infinite-diffeo}) contains only finitely many diffeomorphism types. Namely, suppose that there exists a positive integer $m_{0}$ such that $Z^{{\ell}_{1}, {\ell}_{2}}_{g, h}(m_{0})$ is diffeomorphic to $Z^{{\ell}_{1}, {\ell}_{2}}_{g, h}(m)$ for any integer ${m} \geq {m}_{0}$. Then, by taking ${m} \rightarrow \infty$, we see that the set of monopole classes of $4$-manifold $Z^{{\ell}_{1}, {\ell}_{2}}_{g, h}(m_{0})$ is unbounded by (\ref{connected-f}). However, this is a contradiction because the set of monopole classes of any given smooth $4$-manifold with $b^{+} > 1$ must be finite by Proposition \ref{mono-bounded}. Therefore, the sequence (\ref{infinite-diffeo}) must contain infinitely many diffeomorphism types. For each $m$, since there are infinitely many sufficiently large integers $g, h, {\ell}_{1}, \ell_{2}$ for which $Z^{{\ell}_{1}, {\ell}_{2}}_{g, h}(m)$ has property $\mathcal R$ by Theorem \ref{main-existence-thm}, we are able to conclude that there are infinitely many sufficiently large integers $g, h, {\ell}_{1}, \ell_{2}$ for which $M^{{\ell}_{1}, {\ell}_{2}}_{g,h}$ has $\mathcal{R}-\infty$-property as desired. 
\end{proof}

%%%%%%%%%%%%%%%%%%%%%%%%%%%%%%%%%%%%%%%%%%%%%%%%%%%%%%%%%%%%%%%%%%%%%%%%%%
\subsection{Einstein case}\label{sub-55}
%%%%%%%%%%%%%%%%%%%%%%%%%%%%%%%%%%%%%%%%%%%%%%%%%%%%%%%%%%%%%%%%%%%%%%%%%%

A Riemannian metric $g$ is called Einstein if its Ricci curvature, considered as a function on the unit tangent bundle, is constant. It is known that any closed oriented Einstein $4$-manifold $X$ satisfies
\begin{eqnarray}\label{HTI}
2\eu(X) - 3|\sign(X)| \geq 0. 
\end{eqnarray}
This inequality is called the Hitchin-Thorpe inequality \cite{hit, thor, be}. In particular, Hitchin \cite{hit} proved that any closed oriented Einstein $4$-manifold satisfying $2\eu(X) = 3|\sign(X)|$ is finitely covered by either $K3$ surface or the 4-torus. Gromov \cite{gromov} improved the Hitchin-Thorpe inequality (\ref{HTI}) by using simplical volume. In fact, Gromov proves that any closed oriented Einstein $4$-manifold $X$ must satisfy 
$$
2\eu(X) - 3|\sign(X)| \geq \frac{1}{1295{\pi}^2}||X||. 
$$
Notice that this is nothing but the inequality (\ref{gh-T}) in Conjecture \ref{conj-1}. It is also known that the simplicial volume term of this inequality can replace by a more larger term as follows: 
\begin{eqnarray}\label{GTHK-in}
2\eu(X) - 3|\sign(X)| \geq \frac{1}{81\pi^2}\|X\|. 
\end{eqnarray}
In \cite{ishi-1}, the first examples of $4$-manifolds satisfying the strict case of the inequality (\ref{GTHK-in}) but not admitting any Einstein metrics. The strategy of the proofs of Theorems \ref{main-AcC} and \ref{main-AcCC} can also be used to prove existence theorem of such $4$-manifolds without Einstein metrics. To state the results, let us introduce the following definition which is similar to Definition \ref{def-property}. 
\begin{defn}\label{def-property-E}
Let $X$ be a closed oriented topological $4$-manifold. We say $X$ has property $\mathcal E$ if $X$ satisfies the following properties. 
\begin{enumerate}
\item $X$ has  $||X|| \not=0$ and satisfies the strict case of the inequality (\ref{GTHK-in}):
\begin{eqnarray*}
2\eu(X) - 3|\sign(X)| > \frac{1}{81{\pi}^2}||X||. 
\end{eqnarray*}
\item $X$ admits at least one smooth structure for which no Einstein metric exists. 
\end{enumerate}
Similarly, we say $X$ has $\mathcal{E}-\infty$ property if $X$ satisfies the above condition 1 and moreover admits infinitely many smooth structures for which no Einstein metric exists. 
\end{defn}

The obstruction to the existence of Einstein metrics proved in \cite{ishi-1} (see Theorem I in \cite{ishi-1}) and the strategy similar with that of proofs of Lemma \ref{cosimplicial-lem-2} and Proposition \ref{prop-FZZ-HTin} immediately imply an Einstein version of Theorem \ref{main-AcC}, where notice that, instead of (\ref{kappa-1}), we need to consider the following quantity:
\begin{eqnarray*}
{\kappa}^{'}(g,h):=  4(1-h)(1-g) - \frac{24(1-h)(1-g)}{81{\pi}^2} > 0. 
\end{eqnarray*}
\begin{main}\label{main-EcC}
Let $X_{m}$ be a BF-admissible closed oriented smooth $4$-manifold and consider the following connected sum:
\begin{eqnarray*}
M^{{\ell}_{1}, {\ell}_{2}}_{g, h, j}:=(\#^{j}_{m=1}X_{m}) \# (\Sigma_{h} \times \Sigma_{g}) \# \ell_{1}({S}^{1} \times {S}^{3}) \# \ell_{2} \overline{{\mathbb C}{P}^{2}}, 
\end{eqnarray*}
where $j = 1, 2$, ${\ell}_{1}, {\ell}_{2} \geq 1$ and $g, h \geq 3$ are odd integers. Then, there are infinitely many sufficiently large integers $g, h, {\ell}_{1}, \ell_{2}$ for which $M^{{\ell}_{1}, {\ell}_{2}}_{g, h, j}$ has property $\mathcal E$. 
\end{main}

Similarly, we also have an Einstein version of Theorem \ref{main-AcCC}:
\begin{main}\label{main-EcCC}
Let $X$ be a BF-admissible closed oriented smooth $4$-manifold and consider the following connected sum:
\begin{eqnarray*}
M^{{\ell}_{1}, {\ell}_{2}}_{g, h}:=X \# K3 \# (\Sigma_{h} \times \Sigma_{g}) \# \ell_{1}({S}^{1} \times {S}^{3}) \# \ell_{2} \overline{{\mathbb C}{P}^{2}}, 
\end{eqnarray*}
where $j = 1, 2$, ${\ell}_{1}, {\ell}_{2} \geq 1$ and $g, h \geq 3$ are odd integers. Then, there are infinitely many sufficiently large integers $g, h, {\ell}_{1}, \ell_{2}$ for which $M^{{\ell}_{1}, {\ell}_{2}}_{g, h}$ has $\mathcal{E}-\infty$ property. 
\end{main}

By Theorems \ref{main-A}, \ref{main-B}, \ref{main-EcC}, and \ref{main-EcCC}, we are able to obtain new examples of non-Einstein 4-manifolds with non-trivial simplicial volume.

%%%%%%%%%%%%%%%%%%%%%%%%%%%%%%%%%%%%%%%%%%%%%%%%%%%%%%%%%%%%%%%%%%%%%%%%%%
\subsection{A generalization of the FZZ conjecture and related results}\label{sub-54}
%%%%%%%%%%%%%%%%%%%%%%%%%%%%%%%%%%%%%%%%%%%%%%%%%%%%%%%%%%%%%%%%%%%%%%%%%%

The FZZ Conjecture claims that if there is a quasi-non-singular solution to the normalized Ricci flow on a closed oriented smooth Riemannian $4$-manifold $X$ with $||X|| \not=0$ and $\bar{\lambda}(X) < 0$, then the following holds:
\begin{eqnarray}\label{GTHK-in-3}
2\eu(X) - 3|\sign(X)| \geq \frac{1}{1295{\pi}^2}||X||. 
\end{eqnarray}
A typical example of non-singular solution of the normalized Ricci flow is an Einstein metric. In fact, any Einstein metric is a fixed point of the normalized Ricci flow. The motivation of Conjecture \ref{conj-1} is coming from the result on Einstein $4$-manifolds. As was already mentioned in the previous section, any Einstein $4$-manifold $X$ must satisfy the inequality (\ref{GTHK-in}). \par
On the other hand, let us recall the definition of the volume entropy (or asymptotic volume) of a Riemannian manifold. Let $X$ be a closed oriented Riemannian manifold with smooth metric $g$, and let $\Tilde{M}$ be its universal cover with the induced metric $\Tilde{g}$. For each $\Tilde{x} \in \Tilde{M}$, let $V(\Tilde{x}, R)$ be the volume of the ball with the center $\Tilde{x}$ and radius $R$. We set 
\begin{eqnarray*}
{\mu}(X, g):=\lim_{R \rightarrow +\infty}\frac{1}{R}\log V(\Tilde{x}, R). 
\end{eqnarray*}
Thanks to work of Manning \cite{man}, it turns out that this limit exists and is independent of the choice of $\Tilde{x}$. It is known \cite{mil} that ${\mu}(X, g) > 0$ if and only if the fundamental group $\pi_{1}({X})$ of $X$ has exceptional growth. We call ${\lambda}(X, g)$ the volume entropy of the metric $g$ and define the volume entropy of $X$ to be
\begin{eqnarray*}
{\mu}(X):=\inf_{g \in {\cal R}^{1}_{X}}{\lambda}(X, g), 
\end{eqnarray*}
where ${\cal R}^{1}_{X}$ means the set of all Riemannian metrics $g$ with unit volume $vol_{g}=1$. This invariant sometimes vanishes even when ${\lambda}(X, g) > 0$ for every $g$. There is a close relationship between simplicial volume and volume entropy of a compact manifold $M$ of dimension $n$ as follows.
$$
\frac{n^{n/2}}{n !}|| M || \geq {\mu}(M)^{n}. 
$$
Now, it is known that any closed Einstein $4$-manifold $X$ must satisfy the following bound (see Introduction of \cite{BIS}):
\begin{eqnarray}\label{gro-vol}
2\eu(X) - 3|\sign(X)| \geq \frac{1}{54{\pi}^2}{\mu}(X)^{4}.  
\end{eqnarray}
The inequality (\ref{GTHK-in-3}) can be derived from this inequality. Hence, the inequality (\ref{gro-vol}) is more stronger than the inequality (\ref{GTHK-in-3}). \par
Based on this result on Einstein case, it is so natural to propose the following conjecture which includes Conjecture \ref{conj-1} as a special case.  
\begin{conj}\label{conj-2}
Let $X$ be a closed oriented smooth Riemannian $4$-manifold with $\mu(X) \not=0$ and $\bar{\lambda}(X) < 0$. Suppose that there is a quasi-non-singular solution to the normalized Ricci flow on $X$. Then the following holds:
\begin{eqnarray}\label{gh-T-2-mu}
2\eu(X) - 3|\sign(X)| \geq \frac{1}{54{\pi}^2}{\mu}(X)^{4}. 
\end{eqnarray}
 \end{conj}
Though this conjecture is completely open, we are able to show that the converse of this conjecture also does not hold in general. To state the main results in this direction, let us introduce
\begin{defn}\label{def-property-mu}
Let $X$ be a closed oriented topological $4$-manifold. We say $X$ has property $\mu$ if $X$ satisfies the following properties. 
\begin{enumerate}
\item $X$ has  ${\mu}(X) \not=0$ and satisfies the strict case of the inequality (\ref{gh-T-2-mu}):
\begin{eqnarray*}
2\eu(X) - 3|\sign(X)| > \frac{1}{54{\pi}^2}{\mu}(X)^{4}. 
\end{eqnarray*}
\item $X$ admits at least one smooth structure for which no for which Perelman's $\bar{\lambda}$ invariant is negative and there is no quasi-non-singular solution to the normalized Ricci flow for any initial metric. 
\end{enumerate}
Similarly, we say $X$ has \textit{$\mu - \infty$ property} if $X$ satisfies the above condition 1 and moreover admits infinitely many smooth structures for which Perelman's $\bar{\lambda}$ invariants are negative and there is no quasi-non-singular solution to the normalized Ricci flow for any initial metric. 
\end{defn}

In \cite{gom}, Gompf showed that, for arbitrary integers $\alpha \geq 2$ and $\beta \geq 0$, one can construct a simply connected symplectic spin $4$-manifold $X_{\alpha, \beta}$ satisfying
\begin{eqnarray}\label{go-1}
\Big( \eu(X_{\alpha, \beta}), \sign(X_{\alpha, \beta}) \Big)= \Big( 24\alpha+4\beta, -16\alpha \Big). 
\end{eqnarray}
 Notice also that this implies
\begin{eqnarray}\label{go-2}
b^+(X_{\alpha, \beta}) &=& 4\alpha+2\beta-1,
\end{eqnarray}
\begin{eqnarray}\label{go-3}
2\eu(X_{\alpha, \beta}) + 3\sign(X_{\alpha, \beta}) &=& 8\beta, 
\end{eqnarray}
\begin{eqnarray}\label{go-4}
2\eu(X_{\alpha, \beta}) - 3\sign(X_{\alpha, \beta}) &=& 8(12\alpha+\beta).
\end{eqnarray}
In what follows, we shall call $X_{\alpha, \beta}$ the Gompf manifold of degree $(\alpha, \beta)$.  Since $b^+(X_{\alpha, \beta}) = 4\alpha+2\beta-1$ by (\ref{go-2}), we have $b^+(X_{\alpha, \beta}) \equiv 3 \ (\bmod \ 4)$ if $4\alpha+2\beta-1 \equiv 3 \ (\bmod \ 4)$ is satisfied. The Gompf manifold $X_{\alpha, \beta}$ is simply connected, we get $b_{1}(X_{\alpha, \beta})=0$. In particular, $X_{\alpha, \beta}$ is BF-admissible in the case where $4\alpha+2\beta-1 \equiv 3 \ (\bmod \ 4)$. \par
Then we have
\begin{lem}\label{mu-lem-2}
Let $X$ be a closed oriented smooth $4$-manifold and consider the following connected sum
\begin{eqnarray*}
M:=X \# X_{\alpha, \beta} \# (\Sigma_{h} \times \Sigma_{g}) \# \ell_{1}({S}^{1} \times {S}^{3}) \# \ell_{2} \overline{{\mathbb C}{P}^{2}}.
\end{eqnarray*}
Then, there are infinitely many integers $\alpha, \beta, g, h, {\ell}_{1}, {\ell}_{2}$ for which the following conditions are satisfied simultaneously: 
\begin{eqnarray}\label{connected-condition-1-mu}
4\alpha+2\beta-1 \equiv 3 \ (\bmod \ 4),  
\end{eqnarray}
\begin{eqnarray}\label{connected-condition-2-mu}
2\eu(X) - 3\sign(X) + 8(12\alpha+\beta)   >  (\frac{128}{27} - 4)(g-1)(h-1)  + 4(j + {\ell}_{1}) -5 \ell_{2},  
\end{eqnarray}
\begin{eqnarray}\label{connected-condition-3-mu}
 2\eu(X) + 3\sign(X) + 8 \beta  > (\frac{128}{27} - 4)(g-1)(h-1) +  4(j + {\ell}_{1}) + \ell_{2},  
\end{eqnarray}
\begin{eqnarray}\label{connected-condition-4-mu}
4(j+\ell_{1}) + \ell_{2} > \frac{1}{3}\Big( 2\eu(X)+ 3\sign(X) +  8 \beta  + 4(1-h)(1-g) \Big).   
\end{eqnarray}
\end{lem}
\begin{proof}
First of all, notice that the inequality (\ref{connected-condition-2-mu}) is always satisfied by taking sufficiently large $\beta$ for any fixed $ \alpha, {\ell}_{1}, {\ell}_{2}, g, h, j$. And notice also that there are infinitely many integers $\alpha, \beta,$ for which (\ref{connected-condition-1-mu}) is satisfied. \par
On the other hand, the inequality (\ref{connected-condition-3-mu}) is equivalent to 
\begin{eqnarray}\label{2-connected-condition-3-mu}
c+8\beta - (\frac{128}{27} - 4)(g-1)(h-1)  > 4(j + {\ell}_{1}) + \ell_{2}. 
\end{eqnarray}
where $c:=2\eu(X) + 3\sign(X)$.
Therefore, by (\ref{connected-condition-4-mu}) and (\ref{2-connected-condition-3-mu}), it is enough to prove that there exist infinitely many positive integers $\alpha, \beta, {\ell}_{1}, {\ell}_{2}, g, h$ satisfying
\begin{eqnarray}\label{3-connected-condition-3-mu}
D > 4(j + {\ell}_{1}) + \ell_{2} > E,  
\end{eqnarray}
where we set as
$$
D := c + 8\beta - (\frac{128}{27} - 4)(g-1)(h-1) , \ E:= \frac{1}{3}\Big(c + 8\beta + 4(1-h)(1-g) \Big). 
$$
Notice that both $D$ and $E$ can become sufficiently large positive integers by taking sufficiently large $\beta$. We also have 
\begin{eqnarray*}
D - E = \frac{2}{3}{c} + \frac{16}{3}{\beta} - (\frac{128}{27} + \frac{4}{3}- 4)(g-1)(h-1). 
\end{eqnarray*}
From this, we see that $D - E$ can become a large positive integer by taking large $\beta$. Since there are infinitely many such a $\beta$, we are able to conclude that there are also infinitely many $\ell_{1}, \ell_{2}$ satisfying $D > 4(2 + {\ell}_{1}) + \ell_{2} > E$. From these observations, we are able to obtain the desired result.
\end{proof}

Let us also recall the following definition due to Gromov \cite{G-fil}.
\begin{defn}
Let $X$ be a connected closed manifold of dimension $n$. $X$ is called essential if there exists a map $X \rightarrow K$ to an aspherical complex $K$ that does not contract to the $(n-1)$-skeleton of $K$. 
\end{defn}

It is known that every simply connected manifold is nonessential. Furthermore, a product of arbitrary manifolds with simply connected manifolds is also nonessential. And it is also known that any nonessential manifold has zero volume entropy. \\
Then we have 
\begin{thm}[Theorem 1.1 in \cite{BIS}]\label{mu-BIS}
Let $X$ and $Y$ be two connected closed oriented manifolds. If $Y$ is nonessential, then $\mu(X\# Y) =\mu(X)$.
\end{thm} 

By Lemma \ref{mu-lem-2} and Theorem \ref{mu-BIS}, we get 

\begin{prop}\label{prop-FZZ-HTin-mu}
Let $X_{m}$ be a nonessential closed oriented smooth $4$-manifold and consider the following connected sum
\begin{eqnarray*}
M:=X \# X_{\alpha, \beta} \# (\Sigma_{h} \times \Sigma_{g}) \# \ell_{1}({S}^{1} \times {S}^{3}) \# \ell_{2} \overline{{\mathbb C}{P}^{2}}.
\end{eqnarray*}
Then, there are infinitely many integers $\alpha, \beta, g, h, {\ell}_{1}, {\ell}_{2}$ for which the following conditions are satisfied simultaneously: 
\begin{eqnarray}\label{connected-condition-1-mu-mu}
4\alpha+2\beta-1 \equiv 3 \ (\bmod \ 4),  
\end{eqnarray}
\begin{eqnarray}\label{Hit-Tho-1-mu}
2\eu(M) - 3|\sign(X)| > \frac{1}{54{\pi}^2}{\mu}(X)^{4} \not= 0,  
\end{eqnarray}
\begin{eqnarray}\label{Hit-Tho-2-mu}
4(j+\ell_{1}) + \ell_{2} > \frac{1}{3}\Big(2\eu(X)+ 3\sign(X) + 8 \beta +4(1-h)(1-g) \Big).  
\end{eqnarray}
\end{prop}
\begin{proof}
First of all, notice that $\overline{{\mathbb C}{P}^{2}}$ and ${S}^{1} \times {S}^{3}$ is nonessential (see also \cite{BIS}). $X$ is also nonessential by the assumption. Therefore, by Theorem \ref{mu-BIS}, we have 
$$
\mu(M) = \mu(\Sigma_{h} \times \Sigma_{g}).
$$
Moreover, Corollary 2.2 in \cite{BIS} tells us that we also have
$$
16(g-1)(h-1) \leq \mu(\Sigma_{h} \times \Sigma_{g})^{4} \leq 256{\pi}^{2}(g-1)(h-1). 
$$ 
Therefore, we obtain
\begin{eqnarray}\label{bound-mu-1}
\frac{16}{54{\pi}^{2}}(g-1)(h-1) \leq \frac{1}{54{\pi}^{2}}\mu(M)^{4} \leq \frac{127}{27}(g-1)(h-1) 
\end{eqnarray}
This particularly tells us that $\mu(M)^{4} \not=0$ whenever $g,h \geq 2$. \par
On the other hand, notice that (\ref{Hit-Tho-2-mu}) is nothing but (\ref{connected-condition-4-mu}). Moreover, by Lemma \ref{simplicial-lem}, we have 
\begin{eqnarray*}
2\eu(M)+3\sign(M) &=& 2\eu(X)+3\sign(X) + 8 \beta +4(g-1)(h-1)\\
&-&4(j+\ell_{1})-{\ell}_{2}. 
\end{eqnarray*}
Therefore, the inequality (\ref{connected-condition-3-mu}) is nothing but 
\begin{eqnarray}\label{Hit-Tho-11-mu}
2\eu(M) + 3\sign(X) > \frac{1}{54{\pi}^2}{\mu}(X)^{4}.
\end{eqnarray}
Similarly, since Lemma \ref{simplicial-lem} also tells us that 
\begin{eqnarray*}
2\eu(M)-3\sign(M) &=& 2\eu(X)-3\sign(X) + 8 (12\alpha + \beta) + 4(g-1)(h-1)\\
&-&4(j+\ell_{1})+5{\ell}_{2}, 
\end{eqnarray*}
the inequality (\ref{connected-condition-2-mu}) is equivalent to 
\begin{eqnarray}\label{Hit-Tho-12-mu}
2\eu(M) - 3\sign(X) >  \frac{1}{54{\pi}^2}{\mu}(X)^{4}. 
\end{eqnarray}
By (\ref{Hit-Tho-11-mu}) and (\ref{Hit-Tho-12-mu}), we obtain (\ref{Hit-Tho-1-mu}). 
\end{proof}

Finally, we obtain 
\begin{main}\label{main-EE-mu}
Let $X$ be a BF-admissible, nonessential closed oriented smooth $4$-manifold, $X_{\alpha, \beta}$ is the Gompf manifold with degree $(\alpha, \beta)$ and consider the following connected sum:
\begin{eqnarray*}
M:=X \# X_{\alpha, \beta} \# (\Sigma_{h} \times \Sigma_{g}) \# \ell_{1}({S}^{1} \times {S}^{3}) \# \ell_{2} \overline{{\mathbb C}{P}^{2}}
\end{eqnarray*}
where ${\ell}_{1}, {\ell}_{2} \geq 1$ and $g, h \geq 3$ are odd integers. And $\alpha \geq 2$ and $\beta \geq 0$. Then, there are infinitely many integers $\alpha, \beta, g, h, {\ell}_{1}, {\ell}_{2}$ for which $M$ has property $\mu$ in the sense of Definition \ref{def-property-mu}. 
\end{main}

\begin{proof}
By Proposition \ref{prop-FZZ-HTin-mu}, there are infinitely many integers $\alpha, \beta, g, h, {\ell}_{1}, {\ell}_{2}$ for which (\ref{connected-condition-1-mu-mu}), (\ref{Hit-Tho-1-mu}) and (\ref{Hit-Tho-2-mu}) hold. Notice that $X_{\alpha, \beta}$ is BF-admissible under (\ref{connected-condition-1-mu-mu}). Since $X$ is also BF-admissible, under (\ref{Hit-Tho-2-mu}), there is no quasi-non-singular solution to the normalized Ricci flow on $M$ for any initial metric by Corollary \ref{speical-ein}. Moreover, we also obtain $\bar{\lambda}(M) < 0$ by Corollary \ref{cor-perel}.
\end{proof}

By considering the sequence of homotopy $K3$ surface used to prove Theorem \ref{main-AcCC}, we also get immediately the following result.
\begin{main}\label{main-EEE-mu}
Consider the following connected sum
\begin{eqnarray*}
M:=K3 \# X_{\alpha, \beta} \# (\Sigma_{h} \times \Sigma_{g}) \# \ell_{1}({S}^{1} \times {S}^{3}) \# \ell_{2} \overline{{\mathbb C}{P}^{2}}
\end{eqnarray*}
where ${\ell}_{1}, {\ell}_{2} \geq 1$ and $g, h \geq 3$ are odd integers. And $\alpha \geq 2$ and $\beta \geq 0$. Then, there are infinitely many integers $\alpha, \beta, g, h, {\ell}_{1}, {\ell}_{2}$ for which $M$ has $\mu -\infty$ property in the sense of Definition \ref{def-property-mu}. 
\end{main}

%\newpage 
\vspace{0.6in}
%%%%%%%%%%%%%%%%%%%%%%%%%%%%

\vfill

{\footnotesize 
\noindent
{R. \.Inan\c{c} Baykur}\\
{Department of Mathematics,  
Brandeis University, USA}\\
{\sc e-mail}: baykur@brandeis.edu \\

{\footnotesize 
\noindent
{Masashi Ishida}\\
{Department of Mathematics,  
Sophia University, Japan }\\
{\sc e-mail}: ishida-m@hoffman.cc.sophia.ac.jp\\

\end{document}